\documentclass[11pt]{amsart}
\usepackage[a4paper,margin=1in]{geometry}
\usepackage{amsmath,amssymb,amsthm,mathrsfs,mathtools}
\usepackage[colorlinks=true,linkcolor=blue,citecolor=magenta,urlcolor=blue]{hyperref}
\usepackage{ bbold }
\usepackage{dirtytalk}
\usepackage{comment}

\renewcommand{\phi}{\varphi}

\newcommand{\Z}{\mathbb{Z}}

\theoremstyle{plain}
\newtheorem{theorem}{Theorem}[section]
\newtheorem{lemma}[theorem]{Lemma}
\newtheorem{proposition}[theorem]{Proposition}
\newtheorem{corollary}[theorem]{Corollary}

\theoremstyle{definition}

\numberwithin{equation}{section}

\title[Arbitrarily Fast Quantum Dispersion in Long-Range Crystals]{Arbitrarily Fast Quantum Dispersion \\ in Long-Range Crystals}

\author{Ga\'etan Leclerc, Mostafa Sabri and Tuomas Sahlsten}

\address{Department of Mathematics and Statistics, University of Helsinki, Helsinki, Finland}
\email{gaetan.leclerc@helsinki.fi}
\email{tuomas.sahlsten@helsinki.fi}

\address{Science Division, New York University Abu Dhabi, Saadiyat Island, Abu Dhabi, UAE.}
\email{mostafa.sabri@nyu.edu}

\thanks{The final phase of this work was concluded during the MFO Oberwolfach event \textit{Novel Developments in Spectral Theory and the Dynamics of Quantum Systems with Applications}, 9 Aug - 14 Aug 2026, we kindly acknowledge the hospitality, inspiring discussions with the other participants and facilities that helped us to complete the work. G.L. and T.S. acknowledge support from the Research Council of Finland's Academy Research Fellowship \emph{``Quantum chaos of large and many body systems''}, grant Nos. 347365, 353738.}

\begin{document}

\begin{abstract}
We construct the first examples of one-dimensional long-range crystals exhibiting arbitrarily fast polynomial dispersive bounds. The Floquet functions of our Hamiltonians are highly oscillatory Weierstrass functions, whose rough autosimilar structure drives the fast dispersion. The proof develops a new Fourier decay theory for $C^\alpha$ images of the Lebesgue measure, based on a Dolgopyat-inspired transfer operator method, and yields a van der Corput lemma for Weierstrass functions. As a consequence, the local time of classical Weierstrass functions of sufficiently large lacunarity exists and is $C^k$, answering a question raised by Geman and Horowitz in 1980. 
\end{abstract}

\maketitle

\section{Introduction}

\subsection{Background and results on dispersion} Quantum transport and dispersion describe complementary aspects of the spreading of solutions $u(t)=e^{-itH}u_0$ of the Schr\"odinger equation $i\partial_tu=Hu$. Transport measures how rapidly the mass of a wave packet moves through space, while dispersion measures its flattening as $t\to\infty$. These phenomena have been extensively studied, especially in one dimension. For Schr\"odinger operators $H$ on $\mathbb Z$, the transport is called \textit{ballistic} if $\|Xe^{-itH}u_0\|_{\ell^2(\mathbb Z)}\asymp t$ as $t\to\infty$ for suitable $u_0\in\ell^2$, where $Xu(n)=nu(n)$ is the position operator. The operator $H$ is called \textit{dispersive} if $\|e^{-itH}\|_{\ell^1(\mathbb Z)\to\ell^\infty(\mathbb Z)}\to0$ as $t\to\infty$. See e.g. \cite{DamanikMalinovitchYoung,BoutetdeMonvelSabri} for background on these.

Transport speed is quite well understood in periodic systems, the motion is ballistic for periodic Schr\"odinger operators and periodic graphs in any dimension, see e.g. \cite{DamanikLukicYessen,FillmanPeriodic,BoutetdeMonvelSabri}. Dispersion in periodic/quasiperiodic systems is more delicate. In one dimension, dispersive estimates go back to Firsova \cite{Firsova} in the continuum, see also \cite{Cu}, while in the discrete setting Mi and Zhao \cite{MiZhao1,MiZhao2} obtained estimates first for period two and then for general periods. More recently, Damanik, Fillman and Young \cite{DamanikFillmanYoung} established the bound $\|e^{-itH}\|_{\ell^1(\mathbb Z)\to\ell^\infty(\mathbb Z)}\lesssim |t|^{-1/3}$ for arbitrary periodic Schr\"odinger operators on $\mathbb Z$, matching the dispersion speed of the discrete Laplacian, while \cite{SKM} proved dispersion for periodic operators on the half-line. In contrast to \cite{DamanikFillmanYoung} where periodic potentials cannot slow down dispersion, it was shown in \cite{BSS} that making the mass periodic in time can significantly slow down dispersion in the Dirac equation, a phenomenon that can also be observed for the adjacency operator of some locally finite crystals \cite[\S\,7.2.4]{KernerPostSabriTaufer}. For quasiperiodic operators, much less is known, with dispersion established for sufficiently small analytic potentials in \cite{BZ20}. For the Fibonacci Hamiltonian, quantum dynamics has been related to the hyperbolic dynamics and fractal geometry of the trace map, see e.g. \cite{DamanikTcheremchantsev,DamanikGorodetski,DamanikGorodetskiYessen} and the recent work \cite{Leclerc2}.

Rather than studying the effect of adding a quasiperiodic potential, a different direction is to relax local finiteness while keeping periodicity. Crystals with summable long-range hopping, described by operators $(Hu)(n)=\sum_{k\in\mathbb Z}w_k u(n+k)$ with $(w_k)_{k\in\mathbb Z}\in\ell^1(\mathbb Z)$ were investigated in \cite{KernerPostSabriTaufer}, motivated by analogies in statistical mechanics and their appearance in the topological classification of insulators. It was shown that allowing infinitely many hopping distances can lead to phenomena impossible in the locally finite setting, including purely singular continuous spectrum, flat bands whose eigenfunctions must have infinite support, ballistic transport without dispersion, and transitions between ballistic and super-ballistic transport. Long-range operators also arise naturally in the study of quasicrystals through \textit{Aubry duality}, which transforms a one-dimensional quasiperiodic Schr\"odinger operator $(H_{V,\alpha,\theta}u)(n)=u(n+1)+u(n-1)+V(\theta+n\alpha)u(n)$ into $(\widehat H u)(n)=\sum_{k\in\mathbb Z}\widehat V(k)u(n-k)+2\cos(2\pi(\theta+n\alpha))u(n)$, see \cite{AvilaJitomirskaya,HaroPuig}. Thus analytic potentials $V$ with infinitely many nonzero Fourier coefficients give non-locally finite hopping on the dual side with a quasiperiodic cosine potential.

The main purpose of this paper is to present a novel technique to study dispersion. We illustrate it here on crystals to show a rather striking behaviour: we construct the first one-dimensional crystals exhibiting \emph{arbitrarily fast polynomial dispersion}, showing that there is no universal polynomial upper bound on the speed of dispersion. The formal definition of a crystal and Floquet function are recalled in \S\,\ref{sec:cryF}. To the best of our knowledge, the fastest dispersion speed previously known for Schr\"odinger operators in one dimension was $t^{-1/2}$.

 The mechanism is perhaps surprising. Previously known dispersive examples had sufficiently regular Floquet functions, allowing classical oscillatory integral estimates such as stationary phase and the standard van der Corput lemma, giving the rates $|t|^{-1/3}$ in the locally finite 1D setting of \cite{DamanikFillmanYoung} and $|t|^{-1/2}$ for the long-range example of \cite{KernerPostSabriTaufer}. Thus one might expect roughness of the Floquet function to obstruct dispersion. Our examples show the opposite: sufficiently autosimilar roughness can produce arbitrarily fast polynomial decay.

Our Floquet functions are highly oscillatory \textit{Weierstrass functions}, with fractal oscillations at every scale. Such structures have appeared in random walks with self-similar long-range jumps and models of anomalous quantum transport \cite{HughesShlesingerMontroll,CaceresNizama}, and are closely connected to hyperbolic dynamics and solenoidal attractors through the study of the Hausdorff dimension of their graphs \cite{BaranskiBaranyRomanowska,Hunt,Ledrappier,Keller,Shen,Tsujii,Romanowska}. Related Weierstrass functions also arise as temporal distance functions for Axiom A systems \cite{TZ23,Leclerc2}. Here we exploit their autosimilarity to prove multiscale cancellation through a transfer-operator argument inspired by Dolgopyat's method \cite{Do98,Na05}. Our result can be equivalently understood as a new $C^\alpha$ van der Corput lemma for Weierstrass functions that may be of independent interest.

Let us start with a special case of our main result. Given \textit{weight vector} $\mathbf w=(w_k)_{k\in\mathbb Z}\in\ell^1(\mathbb Z)$ with $w_{-k}=w_k$ and $w_k\ge 0$, define $H_{\mathbf w} : \ell^2(\Z) \to \ell^2(\Z)$ by
\begin{equation}\label{e:hw}
(H_{\mathbf w}u)(n)=\sum_{k\in\mathbb Z}w_k u(n+k),\qquad u \in \ell^2(\Z), \quad n\in\mathbb Z.
\end{equation}
This is the adjacency operator of a translation-invariant lattice in which $n$ is connected to $n\pm k$ with edge weight $w_k$. Physically speaking, it is essentially the discrete Laplacian of a crystal which is nonlocally finite if infinitely many $w_k$ are nonzero. The standard Laplacian on $\mathbb{Z}$ is a very special case of this construction with $w_{\pm 1}=1$ and $w_k=0$ otherwise. After normalization by $\|\mathbf w\|_{\ell^1}$, the weights may be interpreted as hopping probabilities. 

We prove the following result on dispersion:

\begin{theorem}\label{thm:maincrystal}
For every $N\in\mathbb N$ there exists $\lambda_N\geq2$ such that, for any integer $\lambda\geq\lambda_N$, the Hamiltonian $H_{\mathbf w}$ with $w_{\pm\lambda^j}=2^{-j}$, $j\geq0$, and $w_k=0$ otherwise, satisfies
$$
\|e^{-itH_{\mathbf w}}\|_{\ell^1(\mathbb Z)\to\ell^\infty(\mathbb Z)}
\lesssim_N |t|^{-N},\qquad |t|\to\infty.
$$
\end{theorem}

For example, if we take $\delta_0$ as an initial state and $N\gg 1$, we get a crystal where $\|e^{-itH_{\mathbf{w}}}\delta_0\|_\infty\lesssim |t|^{-N}$, so the state flattens very rapidly with time $t$. 
Regarding quantitative aspects, tracking of the constants in the proof implies a crude estimate $\lambda_N\leq 2^{10^7 N}$, which is far from optimal.

Theorem~\ref{thm:maincrystal} follows from the stronger Theorem~\ref{thm:main} (``$C^\alpha$ van der Corput lemma'') implying dispersion for all regimes of parameters in Weierstrass crystals, which we now discuss in more detail. 

\subsection{From crystals to Fourier decay and a \texorpdfstring{$C^\alpha$}{Calpha} van der Corput lemma}\label{sec:cryF} 
A \emph{crystal} is a graph which is periodic under translations by $\mathbb{Z}^d$. It may or may not be locally finite, and it may carry a periodic potential and/or edge weights. The spectral theory of these crystals is analyzed with Floquet theory. In the present paper, we focus on 1D translation invariant lattices governed by a Hamiltonian $H_{\mathbf{w}}$ as in \eqref{e:hw}, the Floquet theory reduces to the Fourier transform, which diagonalizes $H_{\mathbf{w}}$, showing its equivalence to the multiplication operator on $L^2(\mathbb{T})$ by the \emph{Floquet function}
\begin{equation}\label{e:hthetadef}
h(\theta)=\sum_{k\in\mathbb Z}w_ke^{2\pi ik\theta},\qquad \theta\in\mathbb T.
\end{equation}
For such a function to model a crystal, the edge-weights $w_k$ from $H_{\mathbf{w}}$ should satisfy that 
\begin{equation}\label{e:weights}
    w_k\ge 0, \qquad w_{-k}=w_k \qquad \text{and} \qquad w\in \ell^1(\mathbb{Z}) \,.
\end{equation}
Thus, any crystal with adjacency operator $H_{\mathbf{w}}$ as in \eqref{e:hw} implies rise to a Floquet function $h(\theta)$ as in \eqref{e:hthetadef}-\eqref{e:weights} and conversely, any function $h\in L^2(\mathbb{T})$ whose Fourier coefficients satisfy \eqref{e:weights} defines a crystal with adjacency operator $H_{\mathbf{w}}$, of which it is the Floquet function (meaning that the Fourier coefficients of $h$ become edge weights). We refer e.g. to \cite{KernerPostSabriTaufer} for more background.

The Schr\"odinger evolution becomes multiplication by $e^{-ith(\theta)}$ in Fourier space, giving
\begin{equation}\label{eq:osc-int-intro}
\|e^{-itH_{\mathbf w}}\|_{\ell^1(\mathbb Z)\to\ell^\infty(\mathbb Z)}
=
\sup_{n\in\mathbb Z}
\left|
\int_0^1 e^{2\pi in\theta}e^{-ith(\theta)}\,d\theta
\right| .
\end{equation}
Thus dispersion is equivalent to quantitative decay of these oscillatory integrals, uniformly in $n$. This is the Fourier-analytic mechanism behind Theorem~\ref{thm:maincrystal}.
For $n=0$, the integral is the Fourier transform of the push-forward measure $h_*d\theta$, while for general $n$ it is the Fourier transform of the complex measure $h_*(e^{2\pi in\theta}\,d\theta)$. Hence, dispersion becomes a problem about Fourier decay of images of the Lebesgue measure, uniformly in the Fourier mode.

This places the problem in the broader theory of Fourier decay of dynamically defined measures, which has developed rapidly in recent years (see the survey \cite{Sahlsten}). One direction establishes Fourier decay directly for classes such as Patterson-Sullivan measures \cite{BD,BakerKhalilSahlsten,LequenSahlsten}, self-conformal measures \cite{AlgomHertzWang,AlgomHertzWangLog,BakerSahlsten,BakerKhalilSahlsten,JordanSahlsten,KaufmanCF,QueffelecRamare,SahlstenStevens}, self-affine and stationary measures \cite{Li,LiENS,LiSahlstenAffine,LiSahlsten,LindenstraussVarju,Rapaport}, and equilibrium states and other hyperbolic measures \cite{LeclercBunched,LeclercHyperbolic,LeclercJulia,BakerKhalilSahlsten}, using methods from additive combinatorics and renewal theory to random walks and transfer operators. Another direction, more relevant to us and going back to Kaufman \cite{Kaufman}, exploits smooth nonlinear images of measures, where the change of variables itself creates Fourier decay. Recent results of this type for nonlinear images of self-similar measures were developed in \cite{MosqueraShmerkin,ACWW,BakerBanaji,BanajiYu,BanajiYu2,AlgomBenOvadiaHertzShannon}.

However, the Floquet functions we consider are typically only H\"older, where classical oscillatory integral estimates using derivatives are unavailable and much less is known (see e.g. \cite{Leclerc}). Indeed, for the crystal we study in Theorem~\ref{thm:maincrystal}, the Floquet function \eqref{e:hthetadef} is a \textit{Weierstrass function} $W_{\mu,\lambda}(\theta) = 2\sum_{j=0}^\infty\mu^j\cos (2\pi\lambda^j \theta)$, with $\mu=\frac{1}{2}$. More generally, for non-constant real-analytic and $1$-periodic $\varphi:\mathbb R\to\mathbb R$, $0<\mu<1$ and $\lambda\geq2$, define
\begin{equation}\label{eq:weierstrass-intro}
W_{\mu,\lambda}(\theta)=W_{\mu,\lambda}^{(\varphi)}(\theta)
=
\sum_{j=0}^{\infty}\mu^j\varphi(\lambda^j\theta),\qquad \theta\in\mathbb T^1=\mathbb R/\mathbb Z.
\end{equation}
These are $\beta$-H\"older for every $\beta<\alpha:=|\ln\mu|/\ln\lambda$, and $\alpha$ is typically the maximal exponent of H\"older continuity \cite{Bara}. They satisfy the \textit{autosimilarity relation} 
$$
W_{\mu,\lambda}(\theta)=\varphi(\theta)+\mu W_{\mu,\lambda}(\lambda\theta),\qquad \theta\in\mathbb T^1.
$$
Increasing the lacunarity $\lambda$ makes the Floquet function $W_{\mu,\lambda}$ increasingly rough while preserving its oscillatory structure across scales. This autosimilarity relation with the non-constancy of $\phi$ provide a substitute for differentiability and, through a transfer operator technique that we adapt from the works of Dolgopyat on Anosov flows \cite{Do98,Na05}, allows us to obtain Fourier decay uniformly in the additional Fourier mode in \eqref{eq:osc-int-intro}.

Our main ``$C^\alpha$ van der Corput lemma'' is the following.

\begin{theorem}\label{thm:main}
Let $\varphi:\mathbb R\to\mathbb R$ be a non-constant real-analytic $1$-periodic function, let $0<\mu<1$, and let $\lambda\geq2$ be an integer.
\begin{itemize}
\item[(1)] There exists $\beta=\beta(\lambda,\mu,\varphi)>0$ such that
$$
\sup_{n\in\mathbb Z}
\left|
\int_0^1 e^{2\pi in\theta}e^{itW_{\mu,\lambda}(\theta)}\,d\theta
\right|
\lesssim |t|^{-\beta},
\qquad |t|\to\infty.
$$
\item[(2)]There exist $c(\varphi)>0$ and $\lambda_*(\mu,\varphi)\ge2$ such that, for every integer $\lambda\ge\lambda_*(\mu,\varphi)$:
$$
\sup_{n\in\mathbb Z}
\left|
\int_0^1 e^{2\pi in\theta}e^{itW_{\mu,\lambda}(\theta)}\,d\theta
\right|
\lesssim
|t|^{-c(\varphi)\frac{\ln\lambda}{|\ln\mu|}},
\qquad |t|\to\infty.
$$
\end{itemize}
\end{theorem}

In terms of the maximal H\"older exponent $\alpha=|\ln\mu|/\ln\lambda$, the exponent in part~(2) is of order $1/\alpha$, and this order is optimal by Proposition~\ref{thm:maxdis}. Thus rougher phases $W_{\mu,\lambda}$ give stronger Fourier decay. By \eqref{eq:osc-int-intro}, this implies Theorem~\ref{thm:maincrystal}: for $w_{\pm\lambda^j}=2^{-j}$, taking $\mu=1/2$ and $\varphi(\theta)=2\cos(2\pi\theta)$ implies $W_{\mu,\lambda}=h$ and $\alpha=\ln2/\ln\lambda$, so the dispersive exponent is of order $\ln\lambda/\ln2$ and can be made arbitrarily large as the lacunarity $\lambda$ grows. 

In the special case when $W_{\mu,\lambda}\in C^{1+\alpha}$, equivalently $\lambda\mu<1$, Lemma~\ref{lem:kahane} implies that the fastest decay exponent one can hope to achieve is $t^{-1/2}$.

\subsection{Implications on local time} 
The $C^\alpha$ van der Corput lemma also has a regularity interpretation. For every $n\in\mathbb Z$, consider the complex push-forward measure $\nu_n=(W_{\mu,\lambda})_*(e^{2\pi in\theta}\,d\theta)$. Theorem~\ref{thm:main} implies polynomial decay of $\widehat{\nu_n}$ uniformly in $n$. Whenever the decay exponent is larger than $k+1$, Fourier inversion implies that $\nu_n$ has a $C^k$ density \cite[\S\,1.7.3]{Ru}. Hence increasingly rough Weierstrass phases can produce increasingly smooth image densities. For $n=0$, the measure $\nu_0=(W_{\mu,\lambda})_*d\theta$ is precisely the occupation measure of the Weierstrass function, and its density (when it exists) is the \emph{occupation density} or \emph{local time} \cite{GemanHorowitz}. We therefore obtain:

\begin{corollary}\label{cor:localtime}
For every nonconstant real analytic $\varphi$ and $0<\mu<1$, the local time of $W_{\mu,\lambda}$ exists and is $C^k$ whenever $\lambda\geq\max\{\lambda_*(\mu,\varphi),\exp((k+1)|\ln\mu|/c(\varphi))\}$. In particular, this includes the classical sine and cosine Weierstrass functions for all sufficiently large $\lambda$.
\end{corollary}

This addresses a classical question of Geman and Horowitz \cite{GemanHorowitz} for Weierstrass functions of sufficiently large lacunarity. They asked whether deterministic nowhere differentiable functions such as Weierstrass functions admit local times. The problem has a substantial history. Already in the 80s K\^ono \cite{Kono86,Kono88} and Bertoin \cite{Bertoin88,Bertoin90} studied occupation measures and level sets of certain deterministic self-affine functions, exhibiting both absolutely continuous and singular behaviour. Local times for Takagi-type functions were also studied recently in \cite{ImkellerPamen}. However, these results do not give local times for the classical Weierstrass functions considered here. For Weierstrass-Cellerier functions, \cite{Buc08,Buc} established singularity of the occupation measure and finiteness of almost every level set in particular cases. In particular, for $\varphi(\theta)=\sin(2\pi\theta)$ the local time of $W_{1/2,2}$ does \textit{not} exist, see also the recent survey \cite{BucSurvey}.

More recently, strong positive results have been obtained in the sense of prevalence and randomness. Random Weierstrass functions with independent random coefficients have $L^2$ local times almost surely \cite{Romanowska}. In \cite{ABK}, regularity results for occupation measures of prevalent H\"older functions were proved, and later \cite{BKM} showed that prevalent Weierstrass functions have an $L^2$ occupation density throughout the parameter range, which is bounded and continuous when $\alpha<1/2$. In contrast, Corollary~\ref{cor:localtime} concerns every prescribed nonconstant analytic $\varphi$, including the classical cosine phase arising from our crystals, but requires sufficiently large lacunarity. The Fourier decay obtained here then yields $C^k$ local times, with arbitrarily large $k$ as the lacunarity increases. This is reminiscent of the theory of Bernoulli convolutions \cite{Solomyak,Hochman,Shmerkin,Varju}, where absolute continuity is expected for all sufficiently large contraction parameters.

The regularity phenomenon behind Corollary~\ref{cor:localtime} also has a probabilistic analogue. Kahane's work on image measures of fractional Brownian motion shows that random functions with Hurst exponent $\alpha$ can produce Fourier decay with arbitrarily large exponent as $\alpha$ becomes small, and consequently increasingly regular local times, see \cite{Kahane}. His argument exploits the Gaussian structure of fractional Brownian motion. In our deterministic setting, cancellation instead comes from the rigid autosimilarity $W(\theta)=\varphi(\theta)+\mu W(\lambda\theta)$ and oscillatory transfer-operator estimates. In this sense, Theorem~\ref{thm:main} provides a pseudorandom substitute for randomness in producing strong Fourier decay and smoothness.

\subsection{Further dispersion results}
Theorem~\ref{thm:maincrystal} naturally raises the question of how far this phenomenon can be pushed. In particular, can a crystal disperse exponentially fast? We show that the answer is no, and this obstruction holds for every crystal and any $\mathbf{w} = (w_k)_{k \in \mathbb Z}$.

\begin{proposition}\label{thm:noexponential-intro}
Let $H_{\mathbf{w}}$ be any crystal as in \eqref{e:hw}, \eqref{e:weights}. Then $\langle \delta_0,e^{-itH_{\mathbf{w}}}\delta_0\rangle$ cannot decay exponentially as $|t|\to\infty$. In particular, exponential dispersion is impossible.
\end{proposition}

This statement can be strongly strengthened if the Floquet function \eqref{e:hthetadef} has some regularity. This allows us to prove the following~:

\begin{proposition}\label{prp:lbreg}
    Let $H_{\mathbf{w}}$ be the adjacency operator of a crystal as in \eqref{e:hw}, \eqref{e:weights}. If the crystal weights $(w_k)_{k\geq0}$ are eventually decreasing, then (i) dispersion cannot be faster than $1/t$, and (ii) the crystal has an interval of purely absolutely continuous spectrum.
\end{proposition}

Point (i) shows that long-range hopping by itself is not the reason for Theorem \ref{thm:maincrystal}: the lacunary arrangement of the interactions is essential. Long-range hopping can in fact completely kill dispersion \cite{KernerPostSabriTaufer}. In case of locally finite crystals, dispersion speed cannot exceed $t^{-1/2}$, see Lemma~\ref{lem:kahane}. 

Finally, we prove the following positive results of dispersion. A \emph{lacunary crystal} is a crystal with weights $w_{\pm \lambda_j}=\mu_j\ge 0$, $w_k=0$ otherwise, where $\frac{\lambda_{j+1}}{\lambda_j}\ge q>1$. The Weierstrass crystal is a special case with $\lambda_j=\lambda^j$ and $\mu_j=\mu^j$. We prove:

\begin{theorem}\label{thm:1.6}
    (i) Any lacunary crystal satisfies the dispersive limit $\|e^{itH_{\mathbf{w}}}\|_{\ell^1\to\ell^\infty}\to 0$, provided that $\sum_j \lambda_j\mu_j<\infty$. (ii) Any locally finite crystal exhibits dispersion at a rate $t^{-\beta}$ for some $\beta\in (0,\frac{1}{2}]$. The same holds for long range crystals with exponentially decaying weights.
\end{theorem}

Theorem~\ref{thm:1.6} follows easily from some novel uniform $C^1$ and $C^{1+\alpha}$ van der Corput estimates given in Theorems~\ref{thm:c1vandercorput} and \ref{thm:c1+alphavandercorput}. They hold under weaker assumptions than previously known, and we believe they may be of broader interest. 

\subsection{Transport and spectrum}\label{sec:transpec}
Let us clarify the relation to spectral type. As our crystals are translation invariant, their spectral type is completely governed by the spectral measure $\nu_0(B)=\langle \delta_0, \mathbf{1}_B(H_{\mathbf{w}})\delta_0\rangle$, where $\mathbf{1}_B$ is the indicator function of a Borel set $B$. Fourier decay of a spectral measure rules out atoms, by Wiener's theorem. This implies absence of eigenvalues for $H_{\mathbf{w}}$. Further information depends on the speed of decay: with a slow decay, the spectral measure can be singularly continuous \cite{SiSC7} or absolutely continuous \cite[\S\,7.2.4]{KernerPostSabriTaufer}, while sufficiently fast decay implies absolute continuity \cite[\S\,1.7.3]{Ru}. Conversely, absolute continuity implies decay of each fixed Fourier mode by the Riemann-Lebesgue lemma, but not the uniform decay \eqref{eq:osc-int-intro} required for dispersion, indeed, \cite{KernerPostSabriTaufer} implies long-range crystals with absolutely continuous spectrum and ballistic transport for which dispersion fails.

We now discuss the specific case of Weierstrass crystals. Since $W_{\mu,\lambda}$ is continuous, $\sigma(H_{\mathbf{w}})=W_{\mu,\lambda}(\mathbb T)$ is a compact interval. For $\varphi(\theta)=2\cos(2\pi\theta)$, its upper endpoint is $2/(1-\mu)$, and for odd $\lambda$ one has $\sigma(H_{\mathbf{w}})=\left[-\frac{2}{1-\mu},\frac{2}{1-\mu}\right]$. For even $\lambda$, determining the lower spectral edge becomes a non-trivial extremal problem for the Weierstrass function. The spectral type is more subtle. A classical property of lacunary Fourier series \cite[Vol.~II, p.~267]{Bary} implies that Weierstrass crystals have no point spectrum for any $\mu$ and $\lambda$. Furthermore, $\sigma(H_{\mathbf{w}})$ is in fact purely absolutely continuous in the following regime of parameters:
\begin{itemize}
    \item[(i)] If $\lambda\mu<1$, then $W_{\lambda,\mu}\in C^1$ and $W_{\lambda,\mu}'$ is again a lacunary Fourier series, so \cite[Vol.~II, p.~267]{Bary} shows again that $W_{\lambda,\mu}'$ cannot vanish on a set of positive Lebesgue measure, and this implies $\sigma(H_{\mathbf{w}})=\sigma_{ac}(H_{\mathbf{w}})$, see \cite[Thm. 6.3]{KernerPostSabriTaufer}.
    \item[(ii)] If $\lambda\mu^{\frac{1+\varepsilon}{c(\varphi)}}\ge 1$, where $\varepsilon>0$ is arbitrary and $c(\varphi)$ is as in our Theorem~\ref{thm:main}(2), then our result implies in particular that the Fourier transform of the spectral measure at $\delta_0$ is in $L^1$, and by \cite[1.7.3]{Ru}, this implies it is absolutely continuous, hence $\sigma(H_{\mathbf{w}})=\sigma_{ac}(H_{\mathbf{w}})$.
\end{itemize}

The spectral type for the remaining regime of parameters is open, see \S\,\ref{sec:openp} for a discussion. 

Let us mention that Theorem~\ref{thm:main} with  $\phi(\theta)=\sin 2\pi\theta$ implies, using \cite{Buc}, that purely singularly continuous spectrum can coexist with dispersive bounds. The corresponding Weierstrass function $W$ is not the Floquet function of a crystal (it would if one could take $\phi(\theta)=\cos 2\pi\theta$ in \cite{Buc}), so this is only a statement for multiplication operators, but we expect a similar phenomenon should arise in Schr\"odinger operators.

Concerning the transport speed, taking the initial state $\delta_0$, the Weierstrass crystal exhibits standard ballistic transport if $\lambda\mu<1$ and super-ballistic transport if $\lambda\mu\ge 1$. More precisely, in the second case, the moment $\|Xe^{itH_{\mathbf{w}}}\delta_0\|_{\ell^2(\mathbb{Z})}$ blows up for any $t\neq 0$. This result appeared before in \cite{CaceresNizama}, we outline a proof for completeness.

First note that $W_{\mu,\lambda}\in W^{1,2}(\mathbb{T})\iff \lambda\mu<1$, as is easily shown. It follows from \cite[Cor. 7.7]{KernerPostSabriTaufer} that the motion is ballistic for $\lambda\mu<1$. Conversely, let $\psi_t = e^{itH_{\mathbf{w}}}\delta_0$. Then $\|X e^{itH_{\mathbf{w}}}\delta_0\|_{\ell^2(\mathbb{Z})}<\infty\iff\sum_n n^2 |\psi_t(n)|^2<\infty$. But $\psi_t(n) = \langle \delta_n, e^{itH_{\mathbf{w}}}\delta_0\rangle =\int_0^1 e^{-2\pi i n\theta} e^{itW_{\lambda,\mu}(\theta)} d\theta=\widehat{\varphi}_t(n)$, where $\varphi_t(\theta)=e^{itW_{\lambda,\mu}(\theta)}$.
In other words, $\|X e^{itH_{\mathbf{w}}}\delta_0\|_{\ell^2(\mathbb{Z})}<\infty \iff \varphi_t(\theta)=e^{itW_{\lambda,\mu}(\theta)}\in W^{1,2}(\mathbb{T})$.
Since $e^{it\theta}$ is a local analytic diffeomorphism, then $\lambda\mu\ge 1\implies W_{\mu,\lambda}\notin W^{1,2}\implies e^{itW_{\mu,\lambda}}\notin W^{1,2}\implies\|Xe^{itH_{\mathbf{w}}}\delta_0\|_{\ell^2(\mathbb{Z})}=\infty$ for all $t\neq 0$.

\subsection{Prospects for the theory}\label{sec:openp}

We conclude with some questions concerning the long-range crystals considered above. The lacunary structure of the hopping suggests several natural extensions, in particular to random and higher-dimensional models.

\medskip

(a) \emph{Long-range crystals without dispersion.}
For Weierstrass-type crystals, how essential is the regularity of the phase $\varphi$ for dispersion? Consider the $1$-periodic step function $H=1$ on $[0,1/2)$ and $H=-1$ on $[1/2,1)$, and set
$
W(\theta)=\sum_{n\geq0}\mu^nH(2^n\theta).
$
Since $H(2^n\theta)$ are independent fair signs under Lebesgue measure, $W_*d\theta$ is precisely the Bernoulli convolution $\nu_\mu$, the law of $\sum_{n\geq0}\pm\mu^n$. Thus the analogy between our Weierstrass image measures and Bernoulli convolutions becomes exact, and arithmetic obstructions to dispersion appear: for Pisot $\mu^{-1}$, the classical results of Erd\"os and Salem give $\widehat{\nu_\mu}(t)\not\to0$. This phase does not itself define a crystal in the sense of \eqref{e:hthetadef}-\eqref{e:weights}, but it raises a natural question: can one perturb or smooth out this construction to obtain a genuinely long-range crystal without dispersion while keeping the arithmetic obstruction? Interestingly, some regularity already changes the phenomenon: in Appendix~\ref{ap:large-lacunarity} we show that if the phase $\varphi$ is continuous and non-constant analytic on some interval, while possibly constant elsewhere, then its Weierstrass occupation measure has a Fourier-decaying weak limit as $\lambda\to\infty$. This suggests a link between regularity and arithmetic structure in determining if dispersion happens.

\medskip

(b) \emph{Random and multiplicative hopping.}
Our construction is deterministic, with hopping supported on the lacunary sequence $\lambda^n$ and amplitudes $\mu^n$. A natural random analogue is $w_{\lambda^n}=\mu^nX_n,$
where $(X_n)$ are independent random variables, giving a random Weierstrass Floquet function. In view of the comparison with fractional Brownian motion discussed above, it is natural to ask whether arbitrarily fast polynomial dispersion holds almost surely when $\lambda$ is sufficiently large compared with $\mu^{-1}$. Another direction is to consider multiplicative amplitudes $w_{\lambda^n}=X_1\cdots X_n$, leading to Floquet functions with a multiplicative-cascade structure. Mandelbrot cascade measures are known to exhibit non-trivial polynomial Fourier decay \cite{ChenLiSuomala}, see also \cite{GarbanVargas}. It would be interesting to determine whether analogous random multiplicative constructions of the phase produce almost-sure fast dispersion, whether the dispersive exponent can be expressed in terms of Lyapunov exponents or moments of the multiplicative process, and whether corresponding phase transitions occur in the quantum dynamics. See e.g. \cite{Romanowska} for some works in this random setting for the local time of randomised Weierstrass functions.

\medskip

(c) \emph{A $C^\alpha$ van der Corput principle for fractal measures.}
Our Fourier-decay theorem concerns images of the Lebesgue measure under the autosimilar Weierstrass map. It is natural to ask whether the following extension holds: \textit{if $\nu$ is a Frostman measure on $\mathbb T$, does the image $(W_{\mu,\lambda})_*\nu$ have power Fourier decay for every $0<\mu<1$ and $\lambda\geq2$?} This question was also posed by Osama Khalil to the first author. A positive answer would complement the fractal van der Corput theorem of Algom, Chang, Wu and Wu \cite{ACWW}, where self-similar measures are mapped by sufficiently non-flat smooth functions. Here the situation is in a sense reversed: the source could be an arbitrary Frostman measure, while the map itself carries the rigid autosimilar structure. A particularly natural test case is a self-similar Cantor measure invariant under multiplication by $\lambda$, adapted to the transfer-operator structure used in our proof.

\medskip

(d) \emph{Even faster dispersion speed.}
 Our results show that 1D crystals can disperse at an arbitrarily high polynomial power, whereas exponential dispersion is universally ruled out. Whether there exist crystals that exhibit super-polynomial but sub-exponential dispersion remains open. 
 The Beurling-Malliavin multiplier theorem \cite{BM} implies the existence of compactly supported functions with sub-exponential Fourier decay, for instance of the form $\exp\left(-c\frac{|t|}{(\ln|t|)^2}\right)$.
By Favard's theorem, probability measures with supports containing infinitely many points can be realized as spectral measures of bounded Jacobi matrices, see \cite{KillipSimon}. For crystals, however, the inverse problem is more restrictive: in the one-vertex translation-invariant case, the spectral measure must have the form $h_*(d\theta)$ for a periodic Floquet function $h$, and dispersion moreover requires estimates that are uniform in the off-diagonal Fourier mode.

Our construction suggests one possible approach. Let $\lambda_n\to\infty$ be a strictly increasing sequence of integers and consider the generalized Weierstrass phase
$$
W(\theta)=\sum_{n=0}^\infty \mu^n\varphi(\lambda^{\langle n\rangle}\theta),
\qquad
\lambda^{\langle n\rangle}:=\lambda_1\cdots\lambda_n.
$$
Increasing the lacunarity with the scale could produce progressively stronger oscillatory cancellation and hence super-polynomial dispersion. Alternatively, perhaps ideas from the Beurling-Malliavin construction can be adapted to produce Floquet functions of long-range crystals (i.e. satisfying \eqref{e:hthetadef}-\eqref{e:weights}) with such decay.

\medskip

(e) \emph{Phase diagram of Weierstrass crystals.} 
One may conjecture that for Weierstrass crystals, the spectrum of $H_{\mathbf{w}}$ is purely absolutely continuous whenever $\lambda\mu\neq 1$ and purely singularly continuous when $\lambda\mu=1$. Only part of this conjecture is currently known as we explained in \S\,\ref{sec:transpec}. What supports the conjecture in the critical case $\lambda\mu=1$ is that, in the case of $\varphi(\theta)=\sin 2\pi\theta$, the multiplication operator by the function $W_{1/2,2}$ has purely singularly continuous spectrum by \cite{Buc}. The question (in the critical case) is whether this remains true for $\varphi(\theta)=\cos 2\pi\theta$, which would correspond to a Floquet function of an actual crystal. Other results of similar flavor are those of \cite{AndersonHousworthPitt}, which imply that the multiplication operator by $h(\theta)=\sum_n a_n \lambda^{-n}\cos 2\pi \lambda^n\theta$ has purely singularly continuous spectrum whenever $a_n\to 0$ and $\sum_n a_n^2=\infty$. This function is not of Weierstrass type, but has the advantage of being the Floquet function of a crystal. See also \cite{Pitt82} for a similar result which applies to super-lacunary crystals.

\medskip

(f) \emph{Higher-dimensional crystals.}
Our construction is one-dimensional, but the basic Fourier-analytic formulation extends naturally to $\mathbb Z^d$. A translation-invariant long-range Hamiltonian on $\ell^2(\mathbb Z^d)$ has a Floquet function $h:\mathbb T^d\to\mathbb R$, and dispersion is governed by
$$
\sup_{k\in\mathbb Z^d}
\left|
\int_{\mathbb T^d}e^{2\pi i k\cdot\theta}e^{-ith(\theta)}\,d\theta
\right|.
$$
It is natural to ask whether there are long-range crystals on $\mathbb Z^d$ with arbitrarily fast polynomial dispersion. Separable Weierstrass phases reduce essentially to the one-dimensional theory, while genuinely higher-dimensional autosimilar phases should require new ideas. The relevant problem is to control, uniformly in $k$, the multiscale geometry of regions where $t\nabla h(\theta)$ is close to $2\pi k$. These resonant regions may have non-trivial geometry and interact across different directions and scales, leading to genuinely multidimensional non-concentration problems.

\subsection*{Organisation of the paper}

The paper is organised as follows. Section~\ref{sec:mainproof} proves the Fourier-decay Theorem~\ref{thm:main} using twisted transfer operators and Dolgopyat-type estimates, and deduces the arbitrarily fast dispersion of Theorem~\ref{thm:maincrystal}. In Section~\ref{sec:maximal-dispersion} we study the maximal possible speed of dispersion, proving that exponential dispersion is impossible for any crystal, that it cannot exceed $t^{-\alpha}$ if the Floquet function is $\alpha$-Hölder at any point, that it cannot exceed $t^{-1}$ if the weights are eventually decaying, and that it cannot exceed $t^{-1/2}$ if the graph is locally finite. In Section~\ref{ap:A} we give a new general $C^1$ van der Corput lemma which implies Theorem~\ref{thm:1.6}, and in Appendix~\ref{ap:large-lacunarity} we study the spectral measure in the large-lacunarity limit.

\section{Proof of Theorems \ref{thm:maincrystal} and Theorem \ref{thm:main}} \label{sec:mainproof}

Theorem \ref{thm:maincrystal} follows by establishing the more general Fourier decay statement Theorem \ref{thm:main} on the Fourier decay of complex images of Lebesgue measure under the Weierstrass functions $W := W_{\mu,\lambda}$. We then aim to prove Theorem \ref{thm:main}.

\subsection{Transfer operators}\label{sec:tran}

The core idea of the proof of Theorem \ref{thm:main} is to reduce Fourier decay to quantitative Dolgopyat's estimates on some twisted transfer operator, which is an estimate coming from the study of the mixing rate of some dynamical systems. For a good introduction to Dolgopyat estimates, the interested reader can consult Naud's work \cite{Na05} or the original article by Dolgopyat \cite{Do98}. In our setting, the autosimilarity relation 
$$W(\theta) = \varphi(\theta)+\mu W(\lambda \theta)$$ 
justifies that we put some focus on the dynamical system \say{multiplication by $\lambda$}, defined on the unit circle. This expanding map has $\lambda$ inverse branches, that we will denote $(g_a)_{a \in \{ 0,\dots, \lambda-1 \}}$. More precisely, for every $\lambda \geq 2$, we introduce the alphabet $\mathcal{A}(\lambda) := \{0,\dots,\lambda-1\}$ and the maps 
\[
g_a: \mathbb{S}^1 \sim [0,1] \rightarrow [a/\lambda,(a+1)/\lambda] \subset \mathbb{S}^1\,, \qquad  g_a(\theta) := \frac{a+\theta}{\lambda}\,.
\]
We will often identify $\mathbb{S}^1$ with $[0,1]/\sim$. The multiplication by $\lambda$ satisfies the following important property: for all $h_1,h_2:\mathbb{S}^1\rightarrow \mathbb{R}$, we have
\begin{equation}\label{e:hhl}
\int_{\mathbb{S}^1} h_1(\lambda \theta) h_2(\theta) d\theta = \int_{\mathbb{S}^1} h_1(\theta) \mathcal{L}_{\lambda}(h_2)(\theta) d\theta ,
\end{equation}
where $\mathcal{L}_\lambda h(\theta) := \lambda^{-1} \sum_{a \in \mathcal{A}(\lambda)} h(g_a(\theta))$. The operator $\mathcal{L}_\lambda : L^1(\mathbb{S}^1) \rightarrow L^1(\mathbb{S}^1)$ is called a \textit{transfer operator} associated with the multiplication by $\lambda$. In particular, $\int \mathcal{L}_\lambda h = \int h$. Iterating this transfer operator yields
\[
\mathcal{L}_\lambda^n(h)(\theta) = \frac{1}{\lambda^n}\sum_{\mathbf{a} \in \mathcal{A}(\lambda)^n} h(g_\mathbf{a}(\theta))
\]
where $\mathbf{a} = a_n\dots a_1 \in \mathcal{A}(\lambda)^n$, and $g_{\mathbf{a}}(\theta) := g_{a_n} \circ \dots \circ g_{a_1}(\theta) = x_\mathbf{a}+ \theta \lambda^{-n} $ for $x_\mathbf{a} = g_\mathbf{a}(0)$. Notice that, if $h$ is $L$-Lipschitz, then
\begin{equation}\label{e:mix}
\mathcal{L}_\lambda^n(h) = \int_0^1 h d\theta + \mathcal{O}(L \lambda^{-n}),
\end{equation}
which proves the \textit{exponential mixing} property of the multiplication by $\lambda$ for Lipschitz functions:
\begin{equation}\label{e:pfr}
\int h_1(\lambda^n \theta) h_2(\theta) d\theta = \int_0^1 h_1 d\theta \int_0^1 h_2 d\theta + \mathcal{O}(\lambda^{-n}).
\end{equation}
This is also called a Perron-Frobenius-Ruelle Theorem. A crucial observation is that the larger $\lambda$ is, the quicker this dynamical system is decorrelating the observables $h_2$ and $h_1(\lambda^n \cdot)$. This plays an important role in the fast dispersion result when $\lambda$ grows. 

To handle dispersion, we need to \say{twist} our transfer operator. For a choice of (analytic) phase $\varphi:\mathbb{S}^1 \rightarrow \mathbb{R} / 2\pi \mathbb{Z}$, we define the \say{twisted transfer operator} $\mathcal{L}_{\lambda,i \phi} : C^1([0,1],\mathbb{C}) \rightarrow C^1([0,1],\mathbb{C})$ by:
$$ \mathcal{L}_{\lambda,i\phi}(h)(\theta) := \frac{1}{\lambda}\sum_{a \in \mathcal{A}(\lambda)} e^{i \phi(g_a(\theta))} h(g_a(\theta)) = \mathcal{L}_\lambda(e^{i \phi} h). $$

We will need to iterate twisted transfer operators with changing phase at each iteration. For a sequence of phases $\Phi = (\phi_n)_{n \geq 0}$, let us then define
$$ \mathcal{L}_{\lambda,i \Phi}^{\langle n \rangle} = \mathcal{L}_{\lambda,i \phi_{n-1}} \circ \dots \circ \mathcal{L}_{\lambda,i \phi_0} .$$ A direct computation yields
$$ \mathcal{L}_{\lambda,i \Phi}^{\langle n \rangle} h = \lambda^{-n} \sum_{\mathbf{a} \in \mathcal{A}^{n}} e^{i S_n \Phi \circ g_\mathbf{a}} h \circ g_\mathbf{a}, $$
where $S_n \Phi(\theta) = \sum_{k=0}^{n-1} \phi_k(\lambda^k \theta)$, so that  $S_n\Phi\circ g_{\mathbf{a}}(\theta) =\sum_{k=0}^{n-1} \phi_k(g_{a_{n-k}\dots a_{1}}(\theta))=\sum_{k=1}^{n} \phi_{n-k} (g_{a_{k}\dots a_{1}} \theta)$. We used here that $\lambda g_a(\theta) = \theta \mod 1$, implying $\lambda^k g_{a_n\dots a_1}\theta = g_{a_{n-k}\dots a_1}\theta \text{ mod }1$. The reversed indexing here reflects the order of composition of the transfer operators. \\

Before heading to the proof of dispersion, let us state the \say{Lasota-Yorke} estimates, which says that most of the oscillations in $\mathcal{L}_{\lambda,i\Phi}^{\langle n \rangle}(h)$ come from the phases.

\begin{lemma}[Lasota-Yorke]\label{lem:ly}
Let $\Phi$ be a sequence of phases satisfying $\|\Phi\|_{C^1} := \sup_n \|\phi_n\|_{C^1} < \infty$. Let $h \in C^{1}([0,1],\mathbb{C})$. Then $ |\mathcal{L}_{\lambda,i\Phi}^{\langle n \rangle}(h)| \leq \|h\|_\infty $, and $$ |(\mathcal{L}_{\lambda,i\Phi}^{\langle n \rangle} h)'| \leq \frac{2}{\lambda} \|\Phi\|_{C^1} \|h\|_\infty + \lambda^{-n} \|h'\|_\infty. $$
\end{lemma}

\begin{proof}
First of all:
$$ |\mathcal{L}_{\lambda,i\Phi}^{\langle n \rangle}(h)| = \Big|\lambda^{-n} \sum_{\mathbf{a} \in \mathcal{A}^{n}} e^{i S_n \Phi \circ g_\mathbf{a}} h \circ g_\mathbf{a} \Big| \leq \lambda^{-n} \sum_{\mathbf{a} \in \mathcal{A}^{n}} |h\circ g_\mathbf{a}| \leq \|h\|_\infty. $$
The second inequality comes from:
\begin{align*}(\mathcal{L}_{\lambda,i\Phi}^{\langle n \rangle} h)'(\theta) &= \lambda^{-n} \sum_{\mathbf{a} \in \mathcal{A}^{n}} i (S_n \Phi \circ g_\mathbf{a})' e^{i S_n \Phi \circ g_\mathbf{a}} h \circ g_\mathbf{a} + \lambda^{-n} \sum_{\mathbf{a} \in \mathcal{A}^{n}} e^{i S_n \Phi \circ g_\mathbf{a}} \lambda^{-n} h' \circ g_\mathbf{a} \\
&\leq \sup_{\mathbf{a} \in \mathcal{A}^n} |(S_n \Phi \circ g_\mathbf{a})'| \cdot \|h\|_\infty + \lambda^{-n} \|h'\|_\infty.
\end{align*}
This implies the desired estimate, as $\lambda g_a(\theta) = \theta \mod 1$, so
$(S_n \Phi \circ g_\mathbf{a})' = \left(\sum_{k=1}^{n} \phi_{n-k} \circ g_{a_{k} \dots a_1}\right)' \leq \sum_{k=1}^{n} |\phi_{n-k}'|_\infty \lambda^{-k} \leq \|\Phi\|_{C^1} \frac{1}{\lambda-1} \leq \frac{2}{\lambda} \|\Phi\|_{C^1}$. 
\end{proof}

\subsection{Dispersion and Dolgopyat}
In this subsection, we reduce the dispersive bound to what we call a \say{Dolgopyat estimate}, which is a contraction bound on some twisted transfer operator with specific phases. We use this to prove Theorem~\ref{thm:main}(1). The case of large $\lambda$ in Theorem~\ref{thm:main}(2) will be treated afterwards in \S\,\ref{sec:fasdi}. 

Let $m \in \mathbb{Z}$ and consider the integral $ \int_0^1 e^{i t W(\theta)} e^{2 i \pi m \theta} d\theta$ appearing in Theorem~\ref{thm:main}. We have, by the autosimilarity relation $W(\theta) = \varphi(\theta)+\mu W(\lambda \theta)$:
$$ \int_0^1 e^{i t W(\theta)} e^{2 i \pi m \theta} d\theta = \int_0^1 e^{i t \left(\varphi(\theta) + \mu W(\lambda \theta)\right) } e^{2 i \pi m \theta} d\theta . $$
We perform the Euclidean division of $m$ by $\lambda$, writing $m=m_0+\lambda m^{(1)}$ with $|m_0|\leq \lambda$. We get, using \eqref{e:hhl},
\begin{align*}
    \int_0^1 e^{i t W(\theta)} e^{2 i \pi m \theta} d\theta &= \int_0^1 e^{i \left( t \varphi(\theta) + 2 \pi m_0 \theta \right) } e^{i t \mu W(\lambda \theta)} e^{2 i \pi m^{(1)} \lambda \theta} d\theta\\
    &= \int_0^1 \mathcal{L}_\lambda \Big(e^{i ( t \varphi + 2 \pi m_0 \cdot ) } \Big)(\theta) e^{i t \mu W(\theta)} e^{2 i \pi m^{(1)} \theta} d\theta  \\
    &= \int_0^1 \mathcal{L}_{\lambda,i \phi_{t,m_0}}(1)(\theta) e^{i t \mu W(\theta)} e^{2 i \pi m^{(1)} \theta} d\theta
    \end{align*}
where $\phi_{t,m}(\theta) := t\varphi(\theta)+2 \pi m \theta$. Iterating the argument $n(t)$ times, for $\mu^{n} t \simeq 1$, yields
\begin{equation} \label{e:tranappear}
\int_0^1 e^{i t W(\theta)} e^{2 i \pi m \theta} d\theta = \int_0^1 \mathcal{L}_{\lambda,i\Phi(t,m)}^{\langle n \rangle}(1) e^{i t \mu^n W(\theta)} e^{2 i \pi m^{(n)} \theta} d\theta,
\end{equation}
for the sequence of phases $\Phi(t,m)_k(\theta) := \phi_{t \mu^k, m_k}(\theta) = t \mu^k \varphi(\theta) + 2 \pi m_k \theta$, where $m=\sum_{j} m_j \lambda^j$ is the decomposition of $m$ in base $\lambda$. Notice in particular that, since $0 \leq |m_k| \leq \lambda-1$, we have $\|\Phi(t,m)\|_{C^1} \leq t \|\varphi\|_{C^1} + 2 \pi \lambda \leq t \cdot 4\pi \lambda \|\varphi\|_{C^1} $ for $t \geq \|\varphi\|_{C^1}^{-1}$.  \\

Cauchy-Schwarz then yields the key bound:
\begin{equation}\label{e:redolgo} \Big| \int_0^1 e^{i t W(\theta)} e^{2 i \pi m \theta} d\theta \Big| \leq \sqrt{\int_0^1 |\mathcal{L}_{\lambda,i \Phi(t,m)}^{\langle n \rangle}(1)|^2 d\theta}.
\end{equation}

Dispersion is then reduced to establishing uniform contraction for this family of twisted transfer operators. This is the so-called Dolgopyat's estimate.

\begin{theorem}[$L^2$-Dolgopyat estimate]
Fix $\varphi$, $\mu$, $\lambda$, and suppose that the Weierstrass function $W$ is not analytic. There exists $\varepsilon_0(\mu,\lambda,\varphi)>0$ such that the following holds. Denoting $n(t) := \lfloor  \frac{\ln(t)}{2|\ln(\mu)|} \rfloor$, for any $t\geq 1$ and $m \in \mathbb{Z}$:
$$ \forall h \in C^1, \ \Big\| \mathcal{L}_{\lambda,i \Phi(t,m)}^{\langle n \rangle}(h) \Big\|_{L^2([0,1])} \lesssim \frac{\|h\|_{C^1}}{ t^{\varepsilon_0}} .$$
\end{theorem}

Let us check that $L^2$-Dolgopyat's estimate implies polynomial dispersion.

\begin{proof}[Proof of Theorem~\ref{thm:main}(1) assuming $L^2$-Dolgopyat.]
Let $\mu \in (0,1)$, $\lambda \in \mathbb{N} \setminus \{0,1\}$, and let $\varphi \in C^\omega(\mathbb{S}^1,\mathbb{R})$ be a non-constant analytic function. If $W$ is analytic, then $W'$ is itself analytic, $1$-periodic, and non-constant. It follows that it cannot concentrate, and our van der Corput Corollary~\ref{cor:a3} applies, yielding
$$ \sup_{m \in \mathbb{Z}} \Big| \int_0^1 e^{ i t W(\theta)} e^{2 i \pi m \theta} d\theta \Big| \leq \frac{C}{t^{1/(9k)}}, $$
where $k$ is the maximal order of vanishing of $W'$.

We can then suppose that $W$ is not analytic. In this case, by \eqref{e:redolgo}, we can write
$$ \Big| \int_0^1 e^{ i t W(\theta)} e^{2 i \pi m \theta} d\theta \Big| \leq \Big(\int_0^1 |\mathcal{L}_{\lambda,i\Phi(t,m)}^{\langle n \rangle}(1)|^2 d\theta\Big)^{1/2}. $$
The $L^2$-Dolgopyat estimate then yields
$$ \Big(\int_0^1 |\mathcal{L}_{\lambda,i\Phi(t,m)}^{\langle n \rangle}(1)|^2 d\theta\Big)^{1/2} \leq \frac{C}{t^{\varepsilon_0}} $$
for some $\varepsilon_0>0$ depending on $\mu,\lambda,\varphi$, concluding the proof.
\end{proof}

Fast dispersion when $\lambda$ is large will need more work and is postponed to a bit later. Let us turn into the proof of the $L^2$ bound. The idea is to expand the square modulus in the integral, to find (writing $\Phi := \Phi(t,m)$):
\begin{align*} 
\int_0^1 |\mathcal{L}_{\lambda,i\Phi}^{\langle n \rangle}(h)(\theta)|^2 d\theta &= \int_0^1 \Big| \lambda^{-n} \sum_{\mathbf{a}} e^{i S_n \Phi \circ g_{\mathbf{a}}(\theta)} h(g_\mathbf{a}(\theta)) \Big|^2 d\theta \\
&= \lambda^{-2n} \sum_{\mathbf{a},\mathbf{b} \in \mathcal{A}^n(\lambda)} \int_0^1 e^{i \big( S_n \Phi \circ g_\mathbf{a}(\theta) - S_n \Phi \circ g_\mathbf{b}(\theta) \big)} h(g_\mathbf{a}(\theta)) \overline{h(g_\mathbf{b}(\theta))} d\theta.
\end{align*}
We are thus reduced to estimating a number of oscillatory integrals again, but the phase function $\Psi_{\mathbf{a},\mathbf{b},n,t}(\theta) = S_n\Phi \circ g_\mathbf{a} - S_n\Phi \circ g_\mathbf{b}$ is now smooth, in contrast to the $W$ we had before. The idea is then to apply the standard van der Corput lemma to this phase. However, this requires some $k$-th derivative to be bounded from below, like $|\Psi_{\mathbf{a},\mathbf{b},n,t}^{(k)}(\theta)|\ge 1$. We thus divide the sum into pairs satisfying this lower bound, and pairs that require further analysis, for which the next two lemmas are needed.

The first lemma will show that there is an infinite sequence of $N$ and some $\mathbf{a},\mathbf{b}\in\mathcal{A}^N$ such that $(t\mu^N)^{-1}|\Psi_{\mathbf{a},\mathbf{b},n,t}^{(k)}(\theta)|\ge c_\lambda$. In particular $(S_n\Phi \circ g_\mathbf{a} - S_n\Phi \circ g_\mathbf{b})^{(k)}$ doesn't vanish. Recall here that $\Phi(t,m)_j(\theta)=t\mu^j \varphi(\theta)+2\pi m_j \theta$, so that $\Phi(t,m)_j^{(k)}(\theta) = t \mu^{j} \varphi^{(k)}(\theta)$ if $k\ge 2$, and $(S_n \Phi \circ g_\mathbf{a})^{(k)} = t \mu^n \sum_{j=1}^{n} { \varphi^{(k)}(g_{a_j \dots a_1}(\theta))}{(\mu \lambda^k)^{-j}}$.

\begin{lemma}[Separation]\label{lem:sepk}
Suppose that $W$ is not analytic. Let $k \geq 2$ be such that $\mu \lambda^k \geq 2$. There exists $c_\lambda>0$ and infinitely many $N \geq 1$, and $\mathbf{a},\mathbf{b} \in \mathcal{A}^N$ such that $$ \forall \theta \in [0,1], \  \Big|\sum_{j=1}^{N} \frac{ \varphi^{(k)}(g_{a_j \dots a_1}(\theta))}{(\mu \lambda^k)^j} - \sum_{j=1}^{N} \frac{ \varphi^{(k)}(g_{b_j \dots b_1}(\theta))}{(\mu \lambda^k)^j}\Big| \geq c_\lambda. $$
%where $\varphi^{(k)}$ the $k$-th derivative of $\varphi$.
\end{lemma}

\begin{proof}
Fix $k \geq 2$ large enough. First of all, notice that the statements with $\forall \theta$ and with $\exists \theta$ are equivalent. Indeed, suppose that there exist $c_\lambda>0$, $N \geq 1$, $\mathbf{a},\mathbf{b} \in \mathcal{A}^N$ and $\theta_0 \in [0,1]$ such that
$$ \Big|\sum_{j=1}^{N} \frac{\varphi^{(k)}(g_{a_j \dots a_1}(\theta_0))}{(\mu \lambda^k)^j} - \sum_{j=1}^{N} \frac{\varphi^{(k)}(g_{b_j \dots b_1}(\theta_0))}{(\mu \lambda^k)^j}\Big| \geq c_\lambda. $$
Then, since the expression is continuous in $\theta$, there exists a small ball $B(\theta_0,\varepsilon)$ such that the same bound holds for all $\theta \in B(\theta_0,\varepsilon)$, replacing $c_\lambda$ by $c_\lambda/2$. Since the derivatives of the above expressions are uniformly bounded in $N$, $\mathbf a$ and $\mathbf b$ (as $\mu\lambda^k>1$), the radius $\varepsilon$ may be chosen depending only on $c_\lambda,\mu,\lambda,\varphi$. Hence the length $N_0$ of the word $\mathbf c$ may also be chosen uniformly. Now choose some word $\mathbf{c} \in \mathcal{A}^{N_0}$ such that $g_\mathbf{c}([0,1]) \subset B(\theta_0,\varepsilon)$. Consider now the words $\mathbf{ac}$ and $\mathbf{bc}$. Since they have the same suffix, a direct cancellation yields, for every $\theta\in[0,1]$:
\begin{align*}
&\left|\sum_{j=1}^{N+N_0}\frac{\varphi^{(k)}(g_{(\mathbf{ac})_j\dots(\mathbf{ac})_1}(\theta))}{(\mu\lambda^k)^j} - \sum_{j=1}^{N+N_0}\frac{\varphi^{(k)}(g_{(\mathbf{bc})_j\dots(\mathbf{bc})_1}(\theta))}{(\mu\lambda^k)^j}\right| \\
&\qquad = \frac{1}{(\mu\lambda^k)^{N_0}}\left|\sum_{j=1}^{N}\frac{\varphi^{(k)}(g_{a_j\dots a_1}(g_{\mathbf c}(\theta)))}{(\mu\lambda^k)^j} - \sum_{j=1}^{N}\frac{\varphi^{(k)}(g_{b_j\dots b_1}(g_{\mathbf c}(\theta)))}{(\mu\lambda^k)^j}\right| \geq \frac{c_\lambda/2}{(\mu\lambda^k)^{N_0}}.
\end{align*}
Therefore, since $g_{\mathbf c}(\theta)\in B(\theta_0,\varepsilon)$, the absolute value of the left-hand side is at least $c_\lambda/(2(\mu\lambda^k)^{N_0})$ for every $\theta\in[0,1]$. This is the statement of Lemma~\ref{lem:sepk} for the words $\mathbf{ac}$ and $\mathbf{bc}$.

Let us then prove this statement with $\exists$ replacing $\forall$, assuming that $W$ is not analytic. We will prove the converse: let us assume that
$$ \forall \varepsilon>0, \forall N \text{ large enough}, \forall \mathbf{a},\mathbf{b} \in \mathcal{A}^N, \forall \theta \in [0,1],\ \Big|\sum_{j=1}^{N} \frac{ \varphi^{(k)}(g_{a_j \dots a_1}(\theta))}{(\mu \lambda^k)^j} - \sum_{j=1}^{N} \frac{ \varphi^{(k)}(g_{b_j \dots b_1}(\theta))}{(\mu \lambda^k)^j}\Big| \leq \varepsilon, $$
and let us show that this implies that $W$ is analytic. Notice first that, being given an infinite word $\mathbf{a} \in \mathcal{A}^\mathbb{N}$ written $\dots a_n \dots a_1$, the series 
\begin{equation}\label{e:Xadef}
X_\mathbf{a}(\theta) := \sum_{j=1}^\infty \frac{\varphi^{(k)}(g_{a_j \dots a_1}(\theta))}{(\mu \lambda^k)^j}
\end{equation}
converges and is analytic in $\theta \in [0,1]$. Indeed, since $\varphi$ is real analytic and periodic, $\varphi^{(k)}$ extends holomorphically to a complex neighbourhood of $\mathbb R$. The inverse branch $g_{a_j\dots a_1}$ contracts imaginary parts by $\lambda^{-j}$, so all the functions $\varphi^{(k)}\circ g_{a_j\dots a_1}$ are holomorphic on a common complex neighbourhood of $[0,1]$ and uniformly bounded there. Since $\mu\lambda^k>1$, the coefficients $(\mu\lambda^k)^{-j}$ are summable, and hence the series converges uniformly on this neighbourhood. Its sum $X_{\mathbf a}$ is therefore holomorphic there, and in particular real analytic on $[0,1]$.

The hypothesis that we made precisely translates to
$$ \forall \mathbf{a},\mathbf{b} \in \mathcal{A}^\mathbb{N}, \forall \theta, \quad X_\mathbf{a}(\theta) = X_{\mathbf{b}}(\theta). $$
Notice that this implies that $X_\mathbf{a}(0)=X_\mathbf{a}(1)$, since then $X_{\dots 00}(0) - X_{\dots 00}(1) = X_{\dots 00}(0) - X_{\dots\lambda-1,\lambda-1}(1)=\sum_{j}{\frac{\varphi^{(k)}(0)}{(\mu \lambda^k)^j}} -\sum_{j}{\frac{\varphi^{(k)}(1)}{(\mu \lambda^k)^j}} = 0$ since $\varphi$ is $1$-periodic, and $g_0(0)=0$ and $g_{\lambda-1}(1)=1$. A similar reasoning shows that $X_\mathbf{a}^{(r)}(0)=X_\mathbf{a}^{(r)}(1)$ for every $r$.

Now, this allows us to prove a cohomology condition on $\varphi^{(k)}$. Indeed, the previous remark allows us to define on the circle the map $\psi_k$ by $\psi_k(\theta) := X_\mathbf{a}(\theta)$, analytic on $\mathbb{S}^1$. For every $b\in\mathcal A$ and every $\theta\in g_b([0,1])$, the definition \eqref{e:Xadef} implies
$$ \mu\lambda^k X_{\mathbf a b}(\lambda\theta-b)=\varphi^{(k)}(\theta)+X_\mathbf a(\theta). $$
Since $X_{\mathbf a b}=X_\mathbf a=\psi_k$ and $\psi_k$ is $1$-periodic, this means that
$$ \varphi^{(k)}(\theta) = \mu \lambda^k \psi_k(\lambda \theta) - \psi_k(\theta). $$
Let us show that this cohomology condition can be integrated to $\varphi^{(k-1)}$. Since $\psi_k$ is analytic and 1-periodic, it admits an analytic 1-periodic antiderivative if and only if $\int_0^1  \psi_k(\theta)d\theta = 0$. But notice that
$$ 0 = \int_0^1 \varphi^{(k)}(\theta)d\theta = (\mu \lambda^{k} -1) \int_0^1\psi_k(\theta) d\theta. $$
Hence $\int_0^1\psi_k(\theta) d\theta = 0$ indeed and our cohomology condition can be integrated into
$$ \varphi^{(k-1)}(\theta)
= \mu \lambda^{k-1} \psi_{k-1}(\lambda\theta)
- \psi_{k-1}(\theta) $$
for some analytic and $1$-periodic $\psi_{k-1}$. We will then iterate this argument until we reach a cohomology condition on $\varphi$. The only possible difficulty would be that there exists some $j$ such that $\mu \lambda^j=1$. In this case, to integrate the relation $$\varphi^{(j)}(\theta) = \psi_j(\lambda \theta) - \psi_j(\theta),$$ just notice that the same relation holds replacing $\psi_j$ with $\psi_j - \int_0^1 \psi_j$ directly, and then we can still integrate. This iterated argument allows us to get the following cohomology condition: there exists some $1$-periodic and analytic map $\psi$ such that
$$ \varphi(\theta) = \mu \psi(\lambda \theta) - \psi(\theta). $$
Indeed, at each integration the coefficient $\mu\lambda^j$ becomes $\mu\lambda^{j-1}$, so after $k$ integrations the coefficient is $\mu$. This implies that $W$ is analytic, since then the series defining $W$ is telescopic and yields $W=-\psi$. The proof is done.
\end{proof}

We now use Lemma~\ref{lem:sepk} inductively to show that on average over $\mathbf{a}\in\mathcal{A}^n$, the function $(t\mu^n)^{-1}(S_n\Phi(t,m)\circ g_{\mathbf{a}})^{(k)}(\theta)$ does not concentrate in any interval:

\begin{lemma}[Tree Lemma]
Suppose that $W$ is not analytic. Fix $k \geq 2$ such that $\mu \lambda^k \geq 2$. Then, for every $\theta \in \mathbb{S}^1$, for all $n \geq 1$, and any interval $I$, we have
$$
\lambda^{-n} \#\Big\{ \mathbf{a} \in \mathcal{A}^n, \frac{1}{t \mu^n}(S_n \Phi(t,m) \circ g_\mathbf{a})^{(k)}(\theta) \in I \Big\} \leq C_\lambda\left( |I|^\gamma + \rho^n_\lambda \right),
$$
where, if $N,c_\lambda$ are given from Lemma~\ref{lem:sepk}, with $N$ large enough so that $2^{-N} \|\varphi^{(k)}\|_\infty \leq c_\lambda/16$, then $\gamma \in (0,1)$ is such that $(1-\lambda^{-N}) (\mu \lambda^k)^{N\gamma} = 1$, $\rho_\lambda=\sqrt[N]{1-\lambda^{-N}}\in (0,1)$, $C_\lambda := (c_\lambda/8)^{-1} \rho^{-N}_\lambda$.
\end{lemma}

The idea of the \say{tree lemma} is that the map $\mathbf{a} \in \{0,1\}^{n} \mapsto(S_n \Phi(t,m) \circ g_\mathbf{a})^{(k)}(\theta)$ looks like the coordinate map of points in a Cantor set. The previous bound is then a sort of dimension estimate on this Cantor set. The natural alphabetical \say{tree structure} of $\{0,1\}^n$ is preserved in some sense and allows to control the non-concentration of these quantities. Taking $k$ large enough ensures that different scales clearly appear in the sum $\sum_j \varphi^{(k)}(\dots) (\lambda^k \mu)^{-j}$.

The parameters in the lemma depend on $\lambda$, this poses no problem here as $\lambda$ is fixed and arbitrary. In the large $\lambda$ regime in \S\,\ref{sec:fasdi} however, we will need a new statement.

\begin{proof}
Denote $c_0=c_\lambda$, $C=C_\lambda$ and $\rho=\rho_\lambda$. First of all, recall that
\begin{equation}\label{e:la}
\frac{1}{t \mu^n}(S_n \Phi(t,m) \circ g_\mathbf{a})^{(k)}(\theta)
= \sum_{j=1}^{n} \frac{\varphi^{(k)}(g_{a_j \dots a_1}(\theta))}{(\mu \lambda^k)^j}
=: \ell_\mathbf{a}(\theta).
\end{equation}
The desired bound on the cardinality is proved by strong induction on $n$. If $n \leq N$, the bound is trivial since
$$
\lambda^{-n} \#\{ \mathbf{a} \in \mathcal{A}^n, \ell_\mathbf{a}(\theta) \in I \}
\leq 1 \leq C\left( |I|^\gamma + \rho^n \right)
$$
as $C \rho^{N} \geq 1$. We can then assume that $n \geq N+1$ and suppose that the result holds for smaller $n$. If $|I| \geq c_0/4$, the bound is also trivial, since
$$
\lambda^{-n} \#\{ \mathbf{a} \in \mathcal{A}^n, \ell_\mathbf{a}(\theta) \in I \}
\leq 1 \leq C\left( |I|^\gamma + \rho^n \right),
$$
as $C (c_0/4)^{\gamma} \geq 1$. Let us then assume that $|I| \leq c_0/4$. We write
$$
\lambda^{-n} \#\{ \mathbf{a} \in \mathcal{A}^n, \ell_\mathbf{a}(\theta) \in I \}
= \lambda^{-N} \sum_{\widehat{\mathbf{a}} \in \mathcal{A}^N }
\lambda^{-(n-N)} \# \{\tilde{\mathbf{a}} \in \mathcal{A}^{n-N},
\ell_{\tilde{\mathbf{a}}\widehat{\mathbf{a}}}(\theta) \in I \}.
$$

Our separation condition ensures that, for the fixed $\theta$, there exists $\widehat{\mathbf a}_0\in\mathcal A^N$ such that
$$
d(\ell_{\widehat{\mathbf a}_0}(\theta),I)\geq c_0/4.
$$
Indeed, if $\mathbf a,\mathbf b$ are given by Lemma~\ref{lem:sepk}, then $|\ell_{\mathbf a}(\theta)-\ell_{\mathbf b}(\theta)|\geq c_0$, so at least one of them has distance at least $c_0/4$ from $I$. Furthermore, a direct computation yields
\begin{equation}\label{e:ell}
\forall \tilde{\mathbf{a}} \in \mathcal{A}^{n-N}, \qquad
\ell_{\tilde{\mathbf{a}}\widehat{\mathbf{a}}}(\theta)
=
\ell_{\widehat{\mathbf{a}}}(\theta)
+
\frac{\ell_{\tilde{\mathbf{a}}}(g_{\widehat{\mathbf{a}}}(\theta))}
{(\mu \lambda^k)^N}.
\end{equation}
Since $\mu\lambda^k\geq2$, uniformly in $\tilde{\mathbf a}$ and $\theta$,
$$
\left|
\frac{\ell_{\tilde{\mathbf{a}}}(g_{\widehat{\mathbf{a}}_0}(\theta))}
{(\mu \lambda^k)^N}
\right|
\leq \frac{\|\varphi^{(k)}\|_\infty}{2^N}
\leq \frac{c_0}{16}.
$$
Thus, by \eqref{e:ell},
$$
d(\ell_{\tilde{\mathbf a}\widehat{\mathbf a}_0}(\theta),I)
\geq \frac{c_0}{4}-\frac{c_0}{16}>0
$$
for every $\tilde{\mathbf a}\in\mathcal A^{n-N}$. Hence no word ending in $\widehat{\mathbf a}_0$ contributes to the cardinality above. Since the constants are independent of $\theta$, this argument is uniform in $\theta$.

This allows us to write, using the induction hypothesis,
\begin{align*}
&\lambda^{-n} \#\{ \mathbf{a} \in \mathcal{A}^n : \ell_\mathbf{a}(\theta) \in I \} \\
&\quad = \lambda^{-N}
\sum_{\widehat{\mathbf{a}} \in \mathcal{A}^N \setminus \{\widehat{\mathbf{a}}_0\}}
\lambda^{-(n-N)}
\#\{ \tilde{\mathbf{a}} \in \mathcal{A}^{n-N} :
\ell_{\tilde{\mathbf{a}}\widehat{\mathbf{a}}}(\theta) \in I \} \\
&\quad = \lambda^{-N}
\sum_{\widehat{\mathbf{a}} \in \mathcal{A}^N \setminus \{\widehat{\mathbf{a}}_0\}}
\lambda^{-(n-N)}
\#\Big\{ \tilde{\mathbf{a}} \in \mathcal{A}^{n-N} :
\ell_{\tilde{\mathbf{a}}}(g_{\widehat{\mathbf{a}}}(\theta))
\in (\mu\lambda^k)^N
\big(I-\ell_{\widehat{\mathbf{a}}}(\theta)\big) \Big\} \\
&\quad \leq \lambda^{-N}
\sum_{\widehat{\mathbf{a}} \in \mathcal{A}^N \setminus \{\widehat{\mathbf{a}}_0\}}
C\Big( |(\mu\lambda^k)^N I|^\gamma + \rho^{n-N} \Big) \\
&\quad = C\big(|I|^\gamma+\rho^n\big),
\end{align*}
since $(1-\lambda^{-N})(\mu\lambda^k)^{N\gamma}
=(1-\lambda^{-N})\rho^{-N}=1$. The induction is done.
\end{proof}

We may now combine the previous lemmas with the van der Corput lemma to prove our $L^2$-Dolgopyat estimates:

\begin{proof}[Proof of $L^2$-Dolgopyat estimates]
We choose $n$ such that $\mu^n t \sim t^{1/2}$. We start by writing, denoting $\Phi=\Phi(t,m)$ for clarity,
\begin{align*}
\Big\| \mathcal{L}_{\lambda,i\Phi(t,m)}^{\langle n \rangle}(h) \Big\|_{L^2([0,1])}^2 &= \int_0^1 |\mathcal{L}_{\lambda,i\Phi(t,m)}^{\langle n \rangle}(h)(\theta)|^2 d\theta \\
&= \int_0^1 \Big| \lambda^{-n} \sum_{\mathbf{a}} e^{i S_n \Phi \circ g_{\mathbf{a}}(\theta)} h(g_\mathbf{a}(\theta)) \Big|^2 d\theta \\
&= \lambda^{-2n} \sum_{\mathbf{a},\mathbf{b} \in \mathcal{A}^n(\lambda)} \int_0^1 e^{i \big( S_n \Phi \circ g_\mathbf{a}(\theta) - S_n \Phi \circ g_\mathbf{b}(\theta) \big)} h(g_\mathbf{a}(\theta)) \overline{h(g_\mathbf{b}(\theta))} d\theta.
\end{align*}

Let us then iterate a little more our transfer operator. Recall that $\int h d\theta=\int \mathcal{L}_\lambda^{\mathfrak n}(h)d\theta$ for any $\mathfrak n$. Fix $k$ large enough so that $\mu \lambda^k \geq 2$. We choose $\mathfrak n$ such that $\lambda^{2\mathfrak n}\simeq t^{1/(4k)}$ and write 
\begin{equation}\label{e:earlier}
\Big\| \mathcal{L}_{\lambda,i\Phi(t,m)}^{\langle n \rangle}(h) \Big\|_{L^2([0,1])}^2 = \frac{1}{\lambda^{2n+2\mathfrak n}} \sum_{\substack{\mathbf a,\mathbf b\in\mathcal A(\lambda)^n\\ \mathbf c\in\mathcal A(\lambda)^{2\mathfrak n}}} \int_0^1 e^{i \big( S_n \Phi \circ g_\mathbf{a}(g_\mathbf{c}(\theta)) - S_n \Phi \circ g_\mathbf{b}(g_\mathbf{c}(\theta)) \big)} h(g_\mathbf{ac}(\theta)) \overline{h(g_\mathbf{bc}(\theta))} d\theta.
\end{equation}
Since $h$ is $C^1$ and $g_{\mathbf{ac}}'=\lambda^{-n-2\mathfrak n}\leq\lambda^{-2\mathfrak n}\lesssim t^{-1/(4k)}$, we have $h(g_{\mathbf{ac}}(\theta))=h(x_{\mathbf{ac}})+\mathcal{O}(\|h\|_{C^1}t^{-1/(4k)})$, and therefore
$$
\Big\| \mathcal{L}_{\lambda,i\Phi(t,m)}^{\langle n \rangle}(h) \Big\|_{L^2([0,1])}^2 \leq \frac{\|h\|_\infty^2}{\lambda^{2n+2\mathfrak n}} \sum_{\substack{\mathbf a,\mathbf b\in\mathcal A(\lambda)^n\\ \mathbf c\in\mathcal A(\lambda)^{2\mathfrak n}}} \left| \int_0^1 e^{i \big( S_n \Phi \circ g_\mathbf{a}(g_\mathbf{c}(\theta)) - S_n \Phi \circ g_\mathbf{b}(g_\mathbf{c}(\theta)) \big)} d\theta \right| + \frac{\|h\|_{C^1}^2}{t^{1/(4k)}}.
$$

Let us concentrate on the integral $I_{\mathbf{a},\mathbf{b},\mathbf{c}}(t):=\int_0^1 e^{i \big( S_n \Phi \circ g_\mathbf{a}(g_\mathbf{c}(\theta)) - S_n \Phi \circ g_\mathbf{b}(g_\mathbf{c}(\theta)) \big)}d\theta$. We write
$$
I_{\mathbf{a},\mathbf{b},\mathbf{c}}(t)=\int_0^1 e^{iT\psi_{\mathbf{a},\mathbf{b},\mathbf{c}}(\theta)}d\theta,
$$
where $T=t\mu^n\lambda^{-(2k+1)\mathfrak n}\simeq t^{(2k-1)/(8k)}$ and
$$
\psi_{\mathbf{a},\mathbf{b},\mathbf{c}}(\theta)=t^{-1}\mu^{-n}\lambda^{(2k+1)\mathfrak n}\big(S_n\Phi\circ g_{\mathbf a}(g_{\mathbf c}(\theta))-S_n\Phi\circ g_{\mathbf b}(g_{\mathbf c}(\theta))\big).
$$
The classical van der Corput lemma implies $|I_{\mathbf{a},\mathbf{b},\mathbf{c}}(t)|\leq c_kT^{-1/k}\lesssim t^{-(2k-1)/(8k^2)}$ whenever $|\psi_{\mathbf{a},\mathbf{b},\mathbf{c}}^{(k)}(\theta)|\geq1$ on $[0,1]$. Hence
\begin{align*}
& \Big\| \mathcal{L}_{\lambda,i\Phi(t,m)}^{\langle n \rangle}(h) \Big\|_{L^2([0,1])}^2 \\
& \lesssim \|h\|_{C^1}^2 \left(t^{-(2k-1)/(8k^2)}+t^{-1/(4k)}+\lambda^{-2n-2\mathfrak n}\#\{(\mathbf a,\mathbf b,\mathbf c)\in\mathcal A^{2n+2\mathfrak n}:\exists\theta,\ |\psi_{\mathbf a,\mathbf b,\mathbf c}^{(k)}(\theta)|\leq1\}\right).
\end{align*}

To conclude, we bound this cardinality. We first compute
\begin{align*}
\psi_{\mathbf{a},\mathbf{b},\mathbf{c}}^{(k)}(\theta) &= \lambda^{\mathfrak n}\big(\ell_{\mathbf a}(g_\mathbf c(\theta))-\ell_{\mathbf b}(g_\mathbf c(\theta))\big) \\
&= \lambda^{\mathfrak n}\big(\ell_{\mathbf a}(g_\mathbf c(0))-\ell_{\mathbf b}(g_\mathbf c(0))\big)+\mathcal O(\lambda^{-\mathfrak n}).
\end{align*}
It follows that, for $t$ large enough, using the tree lemma,
\begin{align*}
&\lambda^{-2n-2\mathfrak n}\#\{(\mathbf a,\mathbf b,\mathbf c)\in\mathcal A^{2n+2\mathfrak n}:\exists\theta,\ |\psi_{\mathbf a,\mathbf b,\mathbf c}^{(k)}(\theta)|\leq1\} \\
&\quad\leq\lambda^{-2n-2\mathfrak n}\#\{(\mathbf a,\mathbf b,\mathbf c)\in\mathcal A^{2n+2\mathfrak n}:\lambda^{\mathfrak n}|\ell_{\mathbf a}(g_\mathbf c(0))-\ell_{\mathbf b}(g_\mathbf c(0))|\leq2\} \\
&\quad\leq\lambda^{-n-2\mathfrak n}\sum_{(\mathbf b,\mathbf c)\in\mathcal A^{n+2\mathfrak n}}\lambda^{-n}\#\{\mathbf a\in\mathcal A^n:\ell_\mathbf a(g_\mathbf c(0))\in[\ell_\mathbf b(g_\mathbf c(0))-2\lambda^{-\mathfrak n},\ell_\mathbf b(g_\mathbf c(0))+2\lambda^{-\mathfrak n}]\} \\
&\quad\leq2C\big(\lambda^{-\gamma\mathfrak n}+\rho^n\big)\lesssim t^{-\varepsilon_0},
\end{align*}
where one may take
$$
\varepsilon_0=\min\left\{\frac{\gamma}{8k},\frac{|\ln\rho|}{2|\ln\mu|}\right\}>0.
$$
Thus
$$
\Big\| \mathcal{L}_{\lambda,i\Phi(t,m)}^{\langle n \rangle}(h) \Big\|_{L^2}^2
\lesssim t^{-\varepsilon_0}\|h\|_{C^1}^2.
$$
Taking square roots and replacing $\varepsilon_0$ by $\varepsilon_0/2$ proves the desired $L^2$-Dolgopyat estimate.
\end{proof}

\subsection{Fast dispersion for large \texorpdfstring{$\lambda$}{lambda}}\label{sec:fasdi}

We proved that dispersion always occurs for Weierstrass crystals, but the decay rate is not very explicit in general. We now turn to the case where $\mu$ and $\varphi$ are fixed, but $\lambda$ is allowed to be large. We will show that in this case, Dolgopyat's estimate can be improved into a bound which can be iterated, and which has a decay exponent independent of $\lambda$. Our enhanced Dolgopyat's estimate involves a parameter $\gamma_{sep}(\varphi)$ defined in terms of the order of vanishing of $\varphi''$, see Lemma~\ref{lem:sep}. For $\varphi(\theta)=2\cos2\pi\theta$, we have $\gamma_{sep}=\frac{1}{6}$.

\begin{theorem}[Dolgopyat's estimate]\label{thm:trueDolgopyat}
Let $\varphi$ be analytic, $1$-periodic, and non-constant. There exist $C_1(\varphi)\geq 1$ and $\gamma_{sep}(\varphi) \in (0,1)$ such that the following holds. Denote $\rho(\varphi) := \gamma_{sep}(\varphi)/132$, let $\mu \in (0,1)$ and denote for $\beta \in (0,1]$, $\lambda_*(\beta) := \max(C_1,\mu^{-500/(\beta \gamma_{sep})})$. 
For every $\beta \in (0,1]$ and $\lambda\geq\lambda_*(\beta)$, there exists $t_*(\beta,\lambda)\geq1$ such that for all $t \geq t_*(\beta,\lambda)$, we have, uniformly in $m$:
$$ \forall h \in C^{1}, \ \Big\| \mathcal{L}_{\lambda,i\Phi(t,m)}^{\langle n_\beta(t) \rangle}(h) \Big\|_{1,{t \mu^{n_\beta(t)}}} \leq \frac{\|h\|_{1,t}}{t^\rho} $$
where $n_\beta(t) := \Big\lfloor \frac{\beta \ln(t)}{|\ln(\mu)|} \Big\rfloor $, so that $\mu^{n_\beta(t)} \simeq t^{-\beta}$, and where $\|h\|_{C^1_t} = \|h\|_{1,t} := \|h\|_\infty+\frac{\|h'\|_{\infty}}{t}$.
\end{theorem}

First of all, let us prove that this implies fast dispersion when $\lambda$ is large.

\begin{proof}[Proof of Theorem \ref{thm:main}(2) assuming Theorem \ref{thm:trueDolgopyat}]
Let $\lambda$ be a large integer (depending on $\varphi$), and let $N$ be the maximal integer such that $\lambda \geq \lambda_*(1/N)$. As $\lambda_*(1/N) = \max(C_1,\mu^{-C_2 N})$ (where $C_2 := \frac{500}{\gamma_{sep}}$), if $\lambda \geq C_1$ we find that $N = \lfloor \frac{\ln \lambda}{C_2(\varphi) \ln \mu^{-1}} \rfloor$. Now choose for all $0 \leq j \leq N-1$ the parameters $\beta_j := \frac{1}{N-j}$ (so that $(1-\frac{j}{N})(1-\beta_j) = 1-\frac{j+1}{N}$). Notice that $\lambda \geq \max_j \lambda_*(\beta_j)$, since $\lambda_*(\beta)=\max(C_1,\mu^{-C_2/\beta})$ is decreasing in $\beta$ and $\beta_j\geq 1/N$. Suppose that $t \geq \max_j t_*(\beta_j,\lambda)^N$. Recall that by \eqref{e:tranappear}, choosing $n := N \lfloor \frac{\ln t}{N|\ln(\mu)|} \rfloor $ (so that $t \simeq \mu^{-n}$), we have
$$ \int_{\mathbb{S}^1} e^{i t W(\theta)} e^{2 i \pi m \theta} d\theta = \int_{\mathbb{S}^1} \mathcal{L}_{\lambda,i \Phi(t,m)}^{\langle n \rangle}(1)(\theta) e^{i t \mu^n W(\theta)} e^{2 i \pi m^{(n)}\theta} d\theta. $$
We then cut the composition in $N$ equal parts, introducing the times $t_j \simeq t \mu^{jn/N} \simeq t^{(1-\frac{j}{N})}$ (satisfying $t_{j} \mu^{n_{\beta_j}(t_{j})}\simeq t_j^{1-\beta_j}=t_{j+1}$) and $t_j \geq t^{1-j/N} \geq t_*(\beta_j,\lambda)$:
$$ \Big| \int_{\mathbb{S}^1} e^{i t W(\theta)} e^{2 i \pi m \theta} d\theta  \Big| \leq \prod_{j=0}^{N-1} \sup_{\tilde m} \big\| \mathcal{L}_{\lambda,i\Phi(t_j,\tilde{m})}^{\langle n_{\beta_j}(t_j)\rangle} \big\|_{\mathcal{L}({C}^{1}_{t_{j}} , {C^{1}_{t_{j+1}}})}  \leq \prod_{j=0}^{N-1} t_j^{-\rho} \leq t^{-\rho \sum_{j=0}^{N-1} (1-j/N)} \leq t^{-\rho N/2}. $$
This proves dispersion at speed $t^{-c(\varphi) \frac{\ln(\lambda)}{|\ln(\mu)|}}$, where $c(\varphi) := \frac{\rho(\varphi)}{2C_2(\varphi)} = \frac{\rho \gamma_{sep}}{1000}=\frac{\gamma_{sep}^2}{132000}$.

For Theorem~\ref{thm:maincrystal}, $\varphi(\theta)=2\cos 2\pi\theta$, we have $\gamma_{sep}=\frac{1}{6}$, hence $c(\varphi)=\frac{1}{4752000} \geq 10^{-7}$. So $t^{-N} \sim t^{-c \frac{\ln\lambda}{\ln 2}}$ implies the order of magnitude $N \sim c \frac{\ln \lambda}{\ln 2}$, so $2^N \sim\lambda^c$ and $\lambda_N \sim 2^{N/c} \sim 2^{10^7 N}$ is enough.
\end{proof}

Let us then turn to the proof of the enhanced Dolgopyat's estimates. We first show that they can be reduced to their $L^2$-version. Crucially, in the next lemma, $\rho$ is independent of $\lambda$.

\begin{lemma}
Suppose that there exists $C_1(\varphi) \geq 1$ and $\gamma_{sep}(\varphi) \in (0,1)$ such that the following holds.
Denote $\rho := \gamma_{sep}/33$ and $\lambda_0(\beta) := \max(C_1(\varphi),\mu^{-2/\beta})$. Suppose that for all $\beta \in (0,1/2]$, for all $\lambda \geq \lambda_0(\beta)$, and for all $t$ large enough depending on $\beta,\lambda$, we have $$ \forall h \in C^1, \  \int_{\mathbb{S}^1} |\mathcal{L}_{\lambda,i\Phi(t,m)}^{\langle n_\beta(t) \rangle} h|^2 d\theta \leq \frac{\|h\|_{1,t}^2}{t^{\rho}} .$$
Then Dolgopyat's estimate holds with decay exponent $\rho/4$, for $\lambda \ge \max( C_1(\varphi),\mu^{-500/ (\beta \gamma_{sep})})$.
\end{lemma}

\begin{proof}

Let $\beta \in (0,1]$. Let $\varepsilon := \rho/5$, and fix $\beta' := \beta(1-\varepsilon)$ and $\beta'':=\beta'/2$. Let $\lambda \geq \max( C_1(\varphi),\mu^{-500/ (\beta \gamma_{sep})}) \geq \max( \mu^{-\frac{15}{\beta \rho}} ,\lambda_0(\beta''))$. For each $\beta''$, let $t_*(\beta'',\lambda)$ be such that the $L^2$ estimate in the hypothesis holds for all $t\geq t_*(\beta'')$. Notice that, since $\varepsilon=\rho/5$ and $\beta'=\beta(1-\varepsilon)$, this choice of $\lambda$ is in particular large enough to satisfy the lower bounds $\lambda\geq\mu^{-3/(\varepsilon\beta)}$ and $\lambda\geq\mu^{-3/\beta'}$ used below. 

First of all, recall that (for $t \geq \|\varphi\|_{C^1}^{-1}$) $$ \|\Phi(t,m)\|_{C^1} \leq 4 \pi \lambda t \|\varphi\|_{C^1}. $$
In particular, Lasota-Yorke (Lemma~\ref{lem:ly}) implies, for any $n$:
$$ \|(\mathcal{L}_{\lambda,i\Phi(t,m)}^{\langle n \rangle}h)'\|_\infty \leq 8 \pi t \| \varphi \|_{C^1} \|h\|_\infty + \lambda^{-n}\| h' \|_\infty$$
and  $\|\mathcal{L}_{\lambda,i\Phi(t,m)}^{\langle n \rangle}(h)\|_\infty \leq \|h\|_\infty$, so that the $\|\cdot \|_{1,t}$ norm can be bounded by
\begin{align*}
\|\mathcal{L}_{\lambda,i\Phi(t,m)}^{\langle n \rangle}(h)\|_{1,t \mu^{n}} &=  \|\mathcal{L}_{\lambda,i\Phi(t,m)}^{\langle n \rangle}(h)\|_{\infty}  + \frac{\|(\mathcal{L}_{\lambda,i\Phi(t,m)}^{\langle n \rangle} h)'\|_\infty}{t \mu^{n}}   \\
& \leq \|h\|_{\infty}  + t^{-1} \mu^{-n} \Big( 8 \pi t \| \varphi \|_{C^1} \|h\|_\infty + \lambda^{-n}\| h' \|_\infty \Big)  \\
& \leq \big(1+8\pi \|\varphi\|_{C^1} \mu^{-n} \big) \|h\|_\infty + \frac{\|h'\|_\infty}{(t \mu^n) \lambda^n} .
\end{align*}
We apply this by writing $\mathcal{L}_{\lambda,i\Phi(t,m)}^{\langle n_\beta(t) \rangle }(h) = \mathcal{L}_{\lambda,i\Phi(t',m')}^{\langle n_{\beta}(t)-n_{\beta'}(t) \rangle}\big( \mathcal{L}_{\lambda,i\Phi(t,m)}^{\langle n_{\beta'}(t) \rangle}(h) \big) $ where $\beta' := \beta(1-\varepsilon)$, $m'$ is some integer, and $t' := t \mu^{n_{\beta(1-\varepsilon)}(t)} \sim t^{1-\beta'}$. We recall here that  $n_\beta(t) := \Big\lfloor \frac{\beta \ln(t)}{|\ln(\mu)|} \Big\rfloor $, so that $\mu^{n_\beta(t)} \simeq t^{-\beta}$. Notice also that $\mu^{n_{\beta}(t)-n_{\beta'}(t)} \simeq t^{-\varepsilon \beta}$. Hence we find:
$$ \|\mathcal{L}_{\lambda,i\Phi(t,m)}^{\langle n_\beta(t) \rangle}(h)\|_{1,t \mu^{n_\beta(t)}} \lesssim t^{\varepsilon \beta} \| \mathcal{L}_{\lambda,i\Phi(t,m)}^{\langle n_{\beta'}(t) \rangle}(h) \|_\infty + \frac{\|(\mathcal{L}_{\lambda,i\Phi(t,m)}^{\langle n_{\beta'}(t) \rangle}h)' \|_\infty}{t^{1-\beta} \lambda^{n_\beta(t) - n_{\beta'}(t)}}. $$
Applying Lasota-Yorke another time on the second term yields:
\begin{align*}\|\mathcal{L}_{\lambda,i\Phi(t,m)}^{{\langle n_\beta(t) \rangle}}(h)\|_{1,t \mu^{n_\beta(t)}} &\lesssim t^{\varepsilon \beta} \| \mathcal{L}_{\lambda,i\Phi(t,m)}^{\langle n_{\beta'}(t) \rangle}(h) \|_\infty + \frac{8 \pi t \| \varphi \|_{C^1} \|h\|_\infty + \lambda^{-n_{\beta'}(t)}\| h' \|_\infty}{t^{1-\beta} \lambda^{n_\beta(t) - n_{\beta'}(t)}} \\
 &\leq t^{\varepsilon \beta} \| \mathcal{L}_{\lambda,i\Phi(t,m)}^{\langle n_{\beta'}(t) \rangle}(h) \|_\infty + \frac{8 \pi t^\beta \| \varphi \|_{C^1} \|h\|_\infty }{ \lambda^{n_\beta(t) - n_{\beta'}(t)}} + t^{\beta} \lambda^{- n_\beta(t)} \frac{ \| h' \|_\infty}{t} .
\end{align*}
Since $n_\beta(t) - n_{\beta'}(t) \simeq \varepsilon \beta \ln(t)/|\ln(\mu)| $, the fact that $\lambda \geq  \mu^{-\frac{3}{\varepsilon \beta} }$ ensures that $\lambda^{n_\beta(t) - n_{\beta'}(t)} \geq t^{2}$ when $t$ is large enough. It follows that, in that case, for $t$ large enough:
\begin{equation}\label{e:star}
\|\mathcal{L}_{\lambda,i\Phi(t,m)}^{\langle n_\beta(t) \rangle}(h)\|_{1,t \mu^{n_\beta(t)}} \lesssim t^{\varepsilon \beta} \| \mathcal{L}_{\lambda,i\Phi(t,m)}^{\langle n_{\beta'}(t) \rangle}(h) \|_\infty + \frac{\|h\|_{1,t}}{t}.
\end{equation}
To reduce to the $L^2$-Dolgopyat estimates, we use Perron-Frobenius-Ruelle in the following way. We have by Cauchy-Schwarz, denoting $\beta'':=\beta'/2$ (for some time $t''$ and integer $m''$):
\[
|\mathcal{L}_{\lambda,i\Phi(t,m)}^{\langle n_{\beta'}(t) \rangle}(h)|^2 = |\mathcal{L}_{\lambda,i\Phi(t'',m'')}^{\langle n_{\beta'}(t)-n_{\beta''}(t) \rangle}( \mathcal{L}_{\lambda,i\Phi(t,m)}^{\langle n_{\beta''}(t) \rangle}(h) )|^2 \leq \mathcal{L}_\lambda^{n_{\beta'}(t)-n_{\beta''}(t)}\Big( \big| \mathcal{L}_{\lambda,i\Phi(t,m)}^{\langle n_{\beta''}(t) \rangle}(h) \big|^2 \Big).
\]
Now, Perron-Frobenius-Ruelle \eqref{e:mix} and Lasota-Yorke yield 
$$ \mathcal{L}_\lambda^{n_{\beta'}(t)-n_{\beta''}(t)}\Big( \big| \mathcal{L}_{\lambda,i\Phi(t,m)}^{\langle n_{\beta''}(t) \rangle}(h) \big|^2 \Big) \lesssim \int_0^1 \big| \mathcal{L}_{\lambda,i\Phi(t,m)}^{\langle n_{\beta''}(t) \rangle}(h) \big|^2 d\theta + \Big\| \big| \mathcal{L}_{\lambda,i\Phi(t,m)}^{\langle n_{\beta''}(t) \rangle}(h) \big|^2 \Big\|_{C^1} \lambda^{-(n_{\beta'}(t)-n_{\beta''}(t))}  $$
$$ \leq \int_0^1 \big| \mathcal{L}_{\lambda,i\Phi(t,m)}^{\langle n_{\beta''}(t) \rangle}(h) \big|^2 d\theta + \|h\|_\infty\Big( \|h\|_\infty+ 2\Big\| \Big(\mathcal{L}_{\lambda,i\Phi(t,m)}^{\langle n_{\beta''}(t) \rangle}(h)\Big)'\Big\|_\infty \Big) \lambda^{-(n_{\beta'}(t)-n_{\beta''}(t))} $$
$$ \leq \int_0^1 \big| \mathcal{L}_{\lambda,i\Phi(t,m)}^{\langle n_{\beta''}(t) \rangle}(h) \big|^2 d\theta + \lambda^{-(n_{\beta'}(t)-n_{\beta''}(t))} \|h\|_\infty \Big( (1+16 \pi t\|\varphi\|_{C^1})\|h\|_\infty + 2\lambda^{-n_{\beta''}(t)}\| h' \|_\infty \Big) $$
$$ \lesssim \int_0^1 \big| \mathcal{L}_{\lambda,i\Phi(t,m)}^{\langle n_{\beta''}(t) \rangle}(h) \big|^2 d\theta + t \lambda^{-(n_{\beta'}(t)-n_{\beta''}(t))} \|h\|_{1,t}^2 $$
$$ \lesssim  \int_0^1 \big| \mathcal{L}_{\lambda,i\Phi(t,m)}^{\langle n_{\beta''}(t) \rangle}(h) \big|^2 d\theta + \frac{\|h\|_{1,t}^2}{t^2} $$
 since $\lambda^{-(n_{\beta'}(t)-n_{\beta''}(t))} \sim t^{- \ln(\lambda) \beta'/|\ln(\mu)| } \lesssim t^{-3}$, as $\lambda \geq \mu^{-3/\beta'}$. Plugging this inequality in \eqref{e:star} yields
$$ \|\mathcal{L}_{\lambda,i\Phi(t,m)}^{\langle n_\beta(t) \rangle}(h)\|_{1,t \mu^{n_\beta(t)}} \lesssim t^{\varepsilon \beta} \Bigg( \Big(\int_0^1 \big| \mathcal{L}_{\lambda,i\Phi(t,m)}^{\langle n_{\beta''}(t) \rangle}(h) \big|^2 d\theta\Big)^{1/2} +  \frac{\|h\|_{1,t}}{t} \Bigg). $$
We conclude using our $L^2$-Dolgopyat hypothesis: since $\lambda \geq \lambda_0(\beta'')$, we have for $t$ large enough:
$$ \|\mathcal{L}_{\lambda,i\Phi(t,m)}^{\langle n_\beta(t) \rangle}(h)\|_{1,t \mu^{n_\beta(t)}} \lesssim t^{\varepsilon \beta} \Big( \frac{ \|h\|_{1,t}}{t^{\rho/2}}  +  \frac{\|h\|_{1,t}}{t} \Big) \leq \frac{\|h\|_{1,t}}{t^{\rho/4}}, $$
since $\varepsilon \beta \leq \rho/5$. The choice $\varepsilon=\rho/5$ is made precisely so that the loss $t^{\varepsilon\beta}$ (and the constant in front) can be absorbed by the gain $t^{-\rho/2}$. 
\end{proof}

Now we turn to the verification of the $L^2$ bound. The strategy is adapted from earlier, but critically the separation lemma and the tree lemma needs to be adapted to a \say{large lambda} version, since the previous bound gets very poor as $\lambda$ grows large. This is the purpose of the next three lemmas, that culminate in a large lambda friendly version of the tree lemma. Our primary goal is to prove a non-concentration condition on the second derivative of $S_n \Phi(t,m) \circ g_\mathbf{a}$. We thus define, for $\mathcal{A}^* := \bigcup_n \mathcal{A}^n$,
$$ \forall \mathbf{a} \in \mathcal{A}^*, \ \ell_\mathbf{a}(\theta) := \sum_{j=1}^{n} \frac{ \varphi^{(2)}(g_{a_j \dots a_1}(\theta))}{(\mu \lambda^2)^{j-1}} ,$$
(taking $k=2$, there is a slight change of normalization from \eqref{e:la}, but we prefer to avoid overloading the notation with $\tilde{\ell}_{\mathbf{a}}$ for example), so that 
\begin{equation}
    \frac{\mu\lambda^2}{t\mu^n}(S_n \Phi(t,m) \circ g_\mathbf{a})'' = \ell_{\mathbf{a}}(\theta) = \varphi^{(2)}(g_{a_1}(\theta)) + \mathcal{O}(\lambda^{-1}) = \ell_{a_1}(\theta)+\mathcal{O}(\lambda^{-1}).
\end{equation}
Since $\lambda$ is chosen large with respect to $\mu$ and $\varphi$, we can assume that {$\mu\lambda^2\ge 2$}, so that the previous series converges uniformly in $\mathbf{a},\theta$. We study the distribution of $(\ell_\mathbf{a}(\theta))_{\mathbf{a} \in \mathcal{A}^*}$.

\begin{lemma}[Non-concentration of periodic analytic maps]\label{lem:analfro}
Let $f:\mathbb{R} \rightarrow \mathbb{R}$ be real analytic, non-constant, and $1$-periodic. Then there exists $\gamma_f \in (0,1)$ and $L_f\in(0,1)$ such that for any interval $I \subset \mathbb{S}^1$ of diameter $\leq L_f$: 
$$ |\{ \theta \in \mathbb{S}^1, \ f(\theta) \in I \}| \leq \frac{1}{10} {|I|^{\gamma_f}}. $$
\end{lemma}

\begin{proof}
For every $\theta_0 \in \mathbb{S}^1$, one can write a local normal form on some small enough ball $B(\theta_0)$ centered at $\theta_0$: $f(\theta) = f(\theta_0)+\big(F_{\theta_0}(\theta-\theta_0)\big)^{k}$ where $F_{\theta_0}$ is a local diffeomorphism. The number $k$ is bounded from above by some $k_{max}$ since $f$ is analytic and not constant (see e.g. \cite[Chap. 4.1]{Don} and use the analytic inverse function theorem). By compactness, we can cover $\mathbb{S}^1$ into a finite number of these balls $B_j$. Then $\theta\in B_j$ implies $y=F_j(\theta-\theta_j)\in F_j(B_j-\theta_j)$ and if $C_j$ is a bound over $(F_j^{-1})'(y)$, this change of variables implies
\begin{align*}
|\{ \theta \in \mathbb{S}^1, \ f(\theta) \in I \}| &= \sum_j |\{ \theta \in B_j, \ f(\theta) \in I \}|\\
&\le\sum_j C_j\left|\{y\in F_j(B_j-\theta_j),\ y^{k_j}\in I-f(\theta_j)\}\right|\le C|I|^{1/k_{max}}.
\end{align*}
For some constant $C$ that only depends on $f$. Now, if $|I|\leq L_f := (10 C)^{-3k_{max}}$, then
$$ |\{ \theta \in \mathbb{S}^1, \ f(\theta) \in I \}|
\leq C |I|^{1/(3k_{max})} |I|^{2/(3k_{max})}
\leq \frac{1}{10}|I|^{2/(3k_{max})}, $$
which is the desired bound for $\gamma_f := 2/(3k_{max})$.
\end{proof}

\begin{lemma}[Separation Lemma]\label{lem:sep}
Suppose that $\varphi$ is analytic, 1-periodic, not constant. There exists $\gamma_{sep}(\varphi) \in (0,1)$ such that, for $\lambda$ large enough $(\geq \max(4(\|\varphi\|_{C^3} L_{\varphi^{(2)}})^{2}, \mu^{-1}) $ is enough, where $L_{\varphi^{(2)}}$ is given by Lemma~\ref{lem:analfro}), for any target function $t:\mathbb{S}^1 \rightarrow \mathbb{R}$, for any $\theta \in \mathbb{S}^1$, there exists a set $D_\theta(t)  \subset \mathcal{A}(\lambda)$ with $\# D_{\theta}(t) \geq (1-\lambda^{-\gamma_{sep}}) \#\mathcal{A}(\lambda)$ such that: 
$$ \forall a \in D_\theta(t), \ \forall \mathbf{b} \in \mathcal{A}^*, \ |\ell_{\mathbf{b}a}(\theta) -t(\theta)| \geq \frac{1}{\sqrt{\lambda}}. $$
\end{lemma}

\begin{proof}
Define 
$$D_\theta(t) := \Big\{a \in \mathcal{A}(\lambda), \ |\ell_a(\theta)-t(\theta)| > \frac{2}{\sqrt{\lambda}}\Big\}.$$
Then indeed, for $\lambda$ large enough, we have for all $a \in D_\theta(t)$ and for all $\mathbf{b} \in \mathcal{A}^*$, $|\ell_{\mathbf{b}a}(\theta)-t(\theta)| \geq \frac{1}{\sqrt{\lambda}}$ since  $ |\ell_{\mathbf{b}a}(\theta) - \ell_a(\theta)| \leq  {\|\varphi^{(2)}\|_{\infty}} \sum_{j=1}^\infty (\mu \lambda^2)^{-j} \leq  \frac{\|\varphi^{(2)}\|_{\infty}}{\lambda}\leq \frac{1}{\sqrt{\lambda}} $ if $\lambda \geq \|\varphi^{(2)}\|_{\infty}^2$, since $\mu \lambda \geq 2$. 
Moreover, as $\varphi$ is Lipschitz, we have when $|\nu-x_a|\leq \frac{1}{\lambda}$, for $\lambda$ large enough ($\geq \|\varphi\|_{C^3}^2 $ is enough):
$$ \mathbb{1}_{[t(\theta)-2\lambda^{-1/2},t(\theta)+2\lambda^{-1/2}]}(\ell_a(\theta)) \leq \mathbb{1}_{[t(\theta)-3\lambda^{-1/2},t(\theta)+3\lambda^{-1/2}]}(\varphi^{(2)}(\nu))  .$$
Integrating in $\nu$ yields 
$$ \lambda^{-1} \mathbb{1}_{[t(\theta)-2\lambda^{-1/2},t(\theta)+2\lambda^{-1/2}]}(\ell_a(\theta)) \leq |\{ \nu \in [x_a,x_a+\lambda^{-1}], \ \varphi^{(2)}(\nu) \in [t(\theta)-3\lambda^{-1/2},t(\theta)+3\lambda^{-1/2}] \}| . $$
We can then bound $\#(\mathcal{A}(\lambda) \setminus D_\theta(t))$ as follows: summing over $a \in \mathcal{A}(\lambda)$ the previous equation,
\begin{multline*} 
\lambda^{-1} \#\Big\{a \in \mathcal{A}(\lambda), \ |\ell_a(\theta)-t(\theta)| \leq \frac{2}{\sqrt{\lambda}}\Big\}  \\ \leq |\{ \nu \in \mathbb{S}^1, \ \varphi^{(2)}(\nu) \in [t(\theta)-3\lambda^{-1/2},t(\theta)+3\lambda^{ -1/2}] \}| \leq \lambda^{-\gamma_{sep}}
\end{multline*}
where $\gamma_{sep} := \gamma_{\varphi^{(2)}}/2 \in (0,1)$ depends only on $\varphi$, by Lemma~\ref{lem:analfro}, taking $\lambda$ sufficiently large. Hence $\#D_\theta(t)\geq (1-\lambda^{-\gamma_{sep}})\#\mathcal{A}(\lambda)$.
\end{proof}

Looking at the proof of the two previous lemmas, we see for $\varphi(\theta)=2\cos(2\pi \theta)$ that one can choose $\gamma_{sep} = \frac{1}{6}$, as announced previously.

\begin{lemma}[Tree lemma]
Let $t:\mathbb{S}^1 \rightarrow \mathbb{R}$ be any target function.
We have, for any $\lambda$ large enough (so that Lemma~\ref{lem:sep} applies), for all $n \geq 1$ and for all $\sigma>0$:
$$ \lambda^{-n} \#\{ \mathbf{a} \in \mathcal{A}(\lambda)^n, \ |\ell_\mathbf{a}(\theta) - t(\theta)|<\sigma \} \leq \lambda (\sigma^{\gamma_{sep}/2}+\lambda^{-\gamma_{sep} n}). $$
where $\gamma_{sep}$ was introduced in Lemma~\ref{lem:sep} and depends only on $\varphi$.
\end{lemma}

Here, as before, the \say{tree structure} should be understood as the \say{self-similarity} relation 
$$\ell_{\mathbf{ab}}(\theta) = \ell_\mathbf{b}(\theta)+ \frac{\ell_\mathbf{a}(g_\mathbf{b}(\theta))}{(\mu \lambda^2)^{|\mathbf{b}|}},$$ which allows to think of the set $(\ell_{\mathbf{a}}(\theta))_{\mathbf{a} \in \mathcal{A}^\mathbb{N}}$ as a (family of) self similar sets in $\mathbb{R}$. The target $t$ will be later chosen of the form $\ell_\mathbf{b}(\theta)$ for some fixed word $\mathbf{b}$.

\begin{proof}
Intuitively, the proof is an induction on the scales $\sigma$. Formally, this is actually an induction on $n$, and then a separation on cases depending on the values of $\sigma.$ \\
When $n=0$ or $n=1$, the bound is clear, since we have the trivial bound $$\lambda^{-n} \#\{ \mathbf{a} \in \mathcal{A}(\lambda)^n, \ |\ell_\mathbf{a}(\theta) - t(\theta)|<\sigma \} \leq 1 \leq \lambda\left(\sigma^{\gamma_{sep}/2}+\lambda^{-n\gamma_{sep}}\right).$$

Now suppose by induction that the result holds for $n$, and let us prove it for $n+1$. If $\sigma \in (\lambda^{-1/2},1)$, the bound is also trivial. So let us assume that $\sigma \leq \lambda^{-1/2}$. Now, notice, using the set $D_\theta(t)$ defined in Lemma~\ref{lem:sep}:
\begin{align*}
&\lambda^{-(n+1)} \#\{ \mathbf{a}b \in \mathcal{A}(\lambda)^{n+1}, \ |\ell_{\mathbf{a}b}(\theta) - t(\theta)|<\sigma \} \\
&\quad = \lambda^{-1} \sum_{b \in \mathcal{A}\setminus D_\theta(t)} \lambda^{-n} \# \{ \mathbf{a} \in \mathcal{A}(\lambda)^{n}, \ |\ell_{\mathbf{a}b}(\theta) - t(\theta)|<\sigma \} \\
&\quad = \lambda^{-1} \sum_{b \in \mathcal{A}\setminus D_\theta(t)} \lambda^{-n} \# \Big\{ \mathbf{a} \in \mathcal{A}(\lambda)^{n}, \ \Big|\frac{\ell_{\mathbf{a}}(g_b(\theta))}{\mu \lambda^2} + \ell_b(\theta) - t(\theta)\Big|<\sigma \Big\} \\
&\quad = \lambda^{-1} \sum_{b \in \mathcal{A}\setminus D_\theta(t)} \lambda^{-n} \# \Big\{ \mathbf{a} \in \mathcal{A}(\lambda)^{n}, \ |\ell_{\mathbf{a}}(g_b(\theta)) + (\ell_b(\theta)-t(\theta))\mu\lambda^2|<\sigma\mu\lambda^2 \Big\} \\
&\quad \leq \lambda^{-1} \sum_{b \in \mathcal{A}\setminus D_\theta(t)} \lambda \left( (\sigma\mu\lambda^2)^{\gamma_{sep}/2} + \lambda^{-n\gamma_{sep}} \right) \\
&\quad \leq \lambda^{-\gamma_{sep}} \cdot \lambda \left( (\sigma\mu\lambda^2)^{\gamma_{sep}/2} + \lambda^{-n\gamma_{sep}} \right) \\
&\quad \leq \lambda \left( \sigma^{\gamma_{sep}/2} + \lambda^{-(n+1)\gamma_{sep}} \right).
\end{align*}
\end{proof}

We are now ready to prove the $L^2$ estimate, using the same strategy as earlier.

\begin{theorem}
Let $\rho=\gamma_{sep}/33$, $\beta \in (0,1/2]$, $\lambda \geq \max(C_1(\varphi),\mu^{-2/\beta})$ and $t$ sufficiently large depending on $\beta,\lambda$. Then we have, uniformly in $m$,
$$
\forall h \in C^1, \qquad
\int_{\mathbb{S}^1} |\mathcal{L}_{\lambda,i\Phi(t,m)}^{\langle n_\beta(t) \rangle} h|^2 d\theta
\leq \frac{\|h\|_{1,t}^2}{t^{\rho}}.
$$
\end{theorem}

\begin{proof}[Proof (of $L^2$-Dolgopyat estimates)]
Let $\beta \in (0,1/2]$ and recall that $n := n_\beta(t)$ is such that $\mu^{n_\beta(t)} \sim t^{-\beta} \geq t^{-1/2}$. We suppose that $\lambda$ is large enough so that the previous tree lemma applies, and that $\lambda \geq \mu^{-2/\beta}.$ Arguing as in \eqref{e:earlier},  denoting $\Phi = \Phi(t,m)$ for clarity, we write:
\begin{align*} 
\Big\| \mathcal{L}_{\lambda,i \Phi(t,m)}^{\langle n \rangle}(h) \Big\|_{L^2([0,1])}^2 
& = \lambda^{-2n} \sum_{\mathbf{a},\mathbf{b} \in \mathcal{A}^n(\lambda)} \int_0^1 e^{i \left( S_n \Phi \circ g_\mathbf{a}(\theta) - S_n \Phi \circ g_\mathbf{b}(\theta) \right)} h(g_\mathbf{a}(\theta)) \overline{h(g_\mathbf{b}(\theta))} d\theta.
\end{align*}
We then iterate a bit more our transfer operator: recall that we have the relation $\int h d\theta = \int \mathcal{L}_\lambda^{\mathfrak{n}}(h) d\theta$ for any $\mathfrak{n}$. We choose $\mathfrak{n}$ such that $\lambda^{2\mathfrak{n}} \simeq t^{1/8}$ and we write:
$$ \Big\| \mathcal{L}_{\lambda,i\Phi(t,m)}^{\langle n \rangle}(h) \Big\|_{L^2([0,1])}^2 = \frac{1}{\lambda^{2n+2 \mathfrak{n}}} \underset{\mathbf{c} \in \mathcal{A}(\lambda)^{2\mathfrak{n}}}{\sum_{\mathbf{a},\mathbf{b} \in \mathcal{A}(\lambda)^n}} \int_0^1 e^{i \big( S_n \Phi \circ g_\mathbf{a}(g_\mathbf{c}(\theta)) - S_n \Phi \circ g_\mathbf{b}(g_\mathbf{c}(\theta)) \big)} h(g_\mathbf{ac}(\theta)) \overline{h(g_\mathbf{bc}(\theta))} d\theta. $$
Since $h$ is $C^1$ and $g_{\mathbf{ac}}' = \lambda^{-n-2\mathfrak{n}} \leq t^{-2}$ for $t$ large enough (as $\lambda \geq \mu^{-2/\beta}$),  this implies, writing $h(g_{\mathbf{ac}}(\theta)) = h(x_{\mathbf{ac}})+\mathcal{O}(\|h\|_{C^1} t^{-2})$:
$$ \Big\| \mathcal{L}_{\lambda,i\Phi(t,m)}^{\langle n \rangle}(h) \Big\|_{L^2([0,1])}^2 \leq \frac{\|h\|_\infty^2}{\lambda^{2n+2\mathfrak{n}}} \underset{\mathbf{c} \in \mathcal{A}(\lambda)^{2\mathfrak{n}}}{\sum_{\mathbf{a},\mathbf{b} \in \mathcal{A}(\lambda)^n}} \Big| \int_0^1 e^{i \big( S_n \Phi \circ g_\mathbf{a}(g_\mathbf{c}(\theta)) - S_n \Phi \circ g_\mathbf{b}(g_\mathbf{c}(\theta)) \big)}  d\theta \Big| + \frac{\|h\|_{{1,t}}^2}{t} .$$
Let us then concentrate on the integral $I_{\mathbf{a},\mathbf{b},\mathbf{c}}(t) := \int_0^1 e^{i \big( S_n \Phi \circ g_\mathbf{a}(g_\mathbf{c}(\theta)) - S_n \Phi \circ g_\mathbf{b}(g_\mathbf{c}(\theta)) \big)}  d\theta$. We have
$$ I_{\mathbf{a},\mathbf{b},\mathbf{c}}(t) =  \int_0^1 e^{i T \psi_{\mathbf{a},\mathbf{b},\mathbf{c}}(\theta)}  d\theta $$
with $T=t \mu^{n-1} \lambda^{-5\mathfrak{n}-2} \geq t^{1/8}$ and 
$$\psi_{\mathbf{a},\mathbf{b},\mathbf{c}}(\theta) = t^{-1} \mu^{1-n} \lambda^{5\mathfrak{n}+2} \big( S_n \Phi \circ g_\mathbf{a}(g_\mathbf{c}(\theta)) - S_n \Phi \circ g_\mathbf{b}(g_\mathbf{c}(\theta)) \big).$$

Again the classical van der Corput lemma ensures that $|I_{\mathbf{a},\mathbf{b},\mathbf{c}}(t)| \leq c T^{-1/2}$ (for some universal constant $c$) if $|\psi_{\mathbf{a},\mathbf{b},\mathbf{c}}^{(2)}(\theta)| \geq 1$ on $[0,1]$. We hence get the bound:
$$ \Big\| \mathcal{L}_{\lambda,i\Phi(t,m)}^{\langle n \rangle}(h) \Big\|_{L^2([0,1])}^2 \lesssim \|h\|_{1,t}^2 \Big( \frac{1}{t^{1/16}} + \lambda^{-2n-2\mathfrak{n}} \#\{ (\mathbf{a},\mathbf{b},\mathbf{c}) \in \mathcal{A}^{2n+2\mathfrak{n}} \ | \ \exists \theta, |\psi_{\mathbf{a},\mathbf{b},\mathbf{c}}^{(2)}(\theta)| \leq 1\} \Big). $$
To conclude, we need to bound this cardinality. We first compute
\begin{align*}
\psi_{\mathbf{a},\mathbf{b},\mathbf{c}}^{(2)}(\theta) &= \lambda^{\mathfrak{n}} \big( \ell_{\mathbf{a}}(g_\mathbf{c}(\theta)) - \ell_{\mathbf{b}}(g_\mathbf{c}(\theta)) \big)  \\
& = \lambda^{\mathfrak{n}} \big( \ell_{\mathbf{a}}(g_\mathbf{c}(0)) - \ell_{\mathbf{b}}(g_\mathbf{c}(0)) \big) + \mathcal{O}(\lambda^{-\mathfrak{n}}) .
\end{align*}
It follows that, for $t$ large enough, we can write, using the tree lemma:
\begin{align*}
& \lambda^{-2n-2\mathfrak{n}} \#\{ (\mathbf{a},\mathbf{b},\mathbf{c}) \in \mathcal{A}^{2n+2\mathfrak{n}} \ | \ \exists \theta, |\psi_{\mathbf{a},\mathbf{b},\mathbf{c}}^{(2)}(\theta)| \leq 1\}\\
&\quad \leq \lambda^{-2n-2\mathfrak{n}} \#\{ (\mathbf{a},\mathbf{b},\mathbf{c}) \in \mathcal{A}^{2n+2\mathfrak{n}} \ , \ \lambda^{\mathfrak{n}} \big| \ell_{\mathbf{a}}(g_\mathbf{c}(0)) - \ell_{\mathbf{b}}(g_\mathbf{c}(0)) \big| \leq 2\} \\
&\quad \leq \lambda^{-n-2\mathfrak{n}} \sum_{(\mathbf{b},\mathbf{c}) \in \mathcal{A}^{n+2\mathfrak{n}}} \lambda^{-n} \#\{ \mathbf{a} \in \mathcal{A}^n, \ \ell_\mathbf{a}(g_\mathbf{c}(0)) \in [\ell_\mathbf{b}(g_\mathbf{c}(0))-2\lambda^{-\mathfrak{n}},\ell_\mathbf{b}(g_\mathbf{c}(0))+2\lambda^{-\mathfrak{n}} ] \} \\
&\quad \leq 2 \lambda \Big( \lambda^{- \gamma_{sep} \mathfrak{n}/2} + \lambda^{-n \gamma_{sep}} \Big) \lesssim t^{-\gamma_{sep}/32}.
\end{align*}
Since also the van der Corput contribution is $\lesssim t^{-1/16} \leq t^{-\gamma_{sep}/32}$, we obtain for $t$ large enough
$$
\int_{\mathbb S^1}
\left|\mathcal L_{\lambda,i\Phi(t,m)}^{\langle n_\beta(t)\rangle}h\right|^2\,d\theta
\leq
t^{-\rho}\|h\|_{1,t}^2,
\qquad
\rho:=\frac{\gamma_{sep}}{33}>0.
$$
Thus $\rho$ depends only on $\varphi$.
\end{proof}

This establishes the desired  Dolgopyat's estimates, and fast dispersion for strongly irregular Weierstrass crystals.

\section{Limits on the rate of dispersion}  \label{sec:maximal-dispersion}
In this section, we complement our main results by showing that dispersive decay cannot be ``too rapid''. A general obstruction is given in \S\,\ref{sec:exp}, and more specific lower bounds are given in \S\,\ref{sec:floreg} when the Floquet function has some weak regularity.

\subsection{Exponential dispersion is impossible}\label{sec:exp}
 In this subsection, we show that dispersive decay at exponential rate cannot occur, this was stated as Proposition~\ref{thm:noexponential-intro}, which we recall here:

\begin{proposition}\label{thm:exponential}
Let $H$ be any crystal on $\mathbb Z$. Then $\langle \delta_0,e^{-itH}\delta_0\rangle$ cannot decay exponentially as $|t|\to\infty$. In particular, exponential dispersion is impossible.
\end{proposition}
\begin{proof}
Suppose on the contrary that $\langle \delta_0,e^{-itH}\delta_0\rangle$ decays exponentially. The proof is a combination of the following steps:
    \begin{enumerate}
        \item Let $\nu_0$ be the spectral measure of the crystal $H$. Then by the spectral theorem, $\langle \delta_0,e^{-itH}\delta_0\rangle = \int_{\sigma(H)} e^{-itx}\,d\nu_0(x) = \widehat{\nu_0}(t)$.
        \item Since exponential decay implies $\widehat{\nu_0}\in L^1(\mathbb R)$, Fourier inversion shows that $\nu_0$ is absolutely continuous with a continuous density $g\in L^1(\mathbb{R})$, see \cite[\S 1.7.3]{Ru}. We thus get that $\langle \delta_0,e^{-itH}\delta_0\rangle = \widehat{\nu_0}(t) = \widehat{g}(t)$, and $|\widehat{g}(t)|\le Ce^{-c|t|}$.
        \item On the other hand, since $H$ is bounded, the spectral density $g$ has a compact support. By the Paley-Wiener theorem, this contradicts the exponential decay of $\widehat{g}(t)$, see \cite[Chapter~4, Corollary 3.2]{SteinShakarchiComplex}.
    \end{enumerate}
    \end{proof}

\subsection{Regularity obstructions to fast dispersion}\label{sec:floreg} 
We now study to what extent the regularity of the Floquet function limits the possible rate of dispersion. We first show that local H\"older regularity already means there is a significant obstruction. More precisely, if the Floquet function is $\alpha$-Hölder at even one point, then the dispersive norm cannot be $o(t^{-1/\alpha})$. 

\begin{proposition}\label{thm:maxdis}
Let $h : \mathbb{R}\rightarrow \mathbb{R}$ be $1$-periodic and continuous. Suppose there exists $a \in \mathbb{S}^1$ such that $h$ is $C^\alpha$ at $a$, in the sense that $|h(a+\theta)-h(a)| \leq C |\theta|^\alpha$, for some $\alpha \in (0,1]$. Then
$$
\exists c_0>0, \exists t_0\geq 1, \forall t \geq t_0, \quad  \sup_{m \in \mathbb{Z}} \Big| \int_0^1 e^{i t h(\theta)} e^{-2 i \pi m \theta} d\theta \Big|  \geq c_0 |t|^{-1/\alpha}.
$$
\end{proposition}

In particular, a crystal whose Floquet function is locally $C^\alpha$ at some point can not disperse strictly faster than $t^{-1/\alpha}$. A strictly faster dispersion can only occur if $h$ is not $\alpha$-Hölder at any point. This complements Theorem~\ref{thm:main}, where increasingly rough Weierstrass Floquet functions produce increasingly fast polynomial dispersion. We also remark that the $\sup_m$ is important, as the proposition would be untrue otherwise by considering $h(\theta)=1$ or $h(\theta)=\theta$.

\begin{proof}[Proof of Proposition \ref{thm:maxdis}]
Suppose, towards a contradiction, that there exists a sequence $t_n \rightarrow \infty$ and $\varepsilon_n \rightarrow 0$ as $n \rightarrow \infty$ such that $$ \sup_{m \in \mathbb{Z}} \Big| \int_0^1 e^{i t_n h(\theta)} e^{-2 i \pi m \theta} d\theta \Big| \leq \varepsilon_n t_n^{-1/\alpha}. $$
Let $\rho \in C^\infty_c(\mathbb{R},\mathbb{R}^+)$ be a smooth nonnegative bump function supported on $[-1/4,1/4]$, with integral one. Choose $a$ a point at which $h$ is $\alpha$-Hölder and denote the associated approximation of $\delta_a$: $\rho_{a,\varepsilon} := \varepsilon^{-1} \rho\big((\cdot-a)/\varepsilon\big)$. First, we have the bound
$$ \Big|\int_{\mathbb{S}^1} e^{it h(\theta)} \rho_{a,\varepsilon}(\theta) d\theta - e^{it h(a)} \Big| \leq  \int_{\mathbb{S}^1} |e^{it h(\theta)} - e^{i t h(a)} | \rho_{a,\varepsilon}(\theta) d\theta \leq t \sup_{\theta \in [a-\varepsilon,a+\varepsilon]}|h(\theta)-h(a)| \lesssim t \varepsilon^\alpha. $$
In particular, $$ \Big|\int_{0}^{1} e^{it_n h(\theta)} \rho_{a,\varepsilon}(\theta) d\theta \Big| \geq 1-C t_n\varepsilon^\alpha.$$
Second, we have:
\begin{align*}
\Big|\int_{0}^{1} e^{it_n h(\theta)} \rho_{a,\varepsilon}(\theta) d\theta \Big|&= \Big|\sum_{m \in \mathbb{Z}} c_m(\rho_{a,\varepsilon}) \int_{-1/2}^{1/2} e^{i t_n h(\theta)} e^{2i\pi m \theta} d\theta \Big|\\
&\lesssim \varepsilon_n t_n^{-1/\alpha} \|c(\rho_{a,\varepsilon})\|_{\ell^1(\mathbb{Z})}
\end{align*}
with 
$$ \|c(\rho_{a,\varepsilon})\|_{\ell^1(\mathbb{Z})} = \sum_{m \in \mathbb{Z}} |\widehat{\rho_{a,\varepsilon}}(m)| = \sum_{\nu \in \varepsilon \mathbb{Z}} |\widehat{\rho}(\nu)| \lesssim \varepsilon^{-1} \|\widehat{\rho}\|_{L^1(\mathbb{R})}. $$
where in the last inequality we used $\varepsilon \sum_{\varepsilon \mathbb{Z}} f(\nu) \rightarrow \int f$, with error $\varepsilon \int |f'|$.
Hence 
$$ \Big|\int_{-1/2}^{1/2} e^{it_n h(\theta)} \rho_{a,\varepsilon}(\theta) d\theta \Big| \lesssim \varepsilon_n t_n^{-1/\alpha} \varepsilon^{-1} .$$
Choosing $\varepsilon :=  (2Ct_n)^{-1/\alpha}$ and letting $n \rightarrow +\infty$ yields a contradiction.
\end{proof}

We next apply this observation to crystals with monotone weights in $H$. Suppose that the weights are eventually decreasing, $w_{n+1}\leq w_n$ for all sufficiently large $n$. We show that this mild assumption already forces the Floquet function to be $C^1$ away from the integers, and hence Proposition~\ref{thm:maxdis} implies a universal $1/t$ obstruction to dispersion.

\begin{lemma}
If $u_n \geq 0$ is eventually decreasing and satisfies $\sum_n u_n < \infty$, then $u_n = o(1/n)$. Moreover, the sequence $\varepsilon_n := n u_n$ satisfies $\sum_{n} |\varepsilon_n-\varepsilon_{n+1}|<\infty$.
\end{lemma}

\begin{proof}
We have $2n u_{2n} \leq u_n+\dots+u_{2n} \leq \sum_{k=n}^{\infty} u_k \rightarrow 0$. Then $n u_n \rightarrow 0$ by monotonicity of $u_n$. Now, letting $\varepsilon_n := n u_n$, We have $$ \sum_{n=0}^N |\varepsilon_n - \varepsilon_{n+1}| = \sum_{n=0}^N |n u_n - (n+1) u_{n+1}| \leq \sum_{n=0}^N n ( u_n -  u_{n+1}) + \sum_{n=0}^N u_{n+1} $$
which converge since, by summation by parts, $$ \sum_{n=0}^N n ( u_n -  u_{n+1}) \leq N u_N + \sum_{n} {u_n} \leq \varepsilon_N + \|u\|_{\ell^1}  $$
is bounded.
\end{proof}

\begin{theorem}\label{thm:monotone-regularity}
Let $(w_n)_{n\in\mathbb Z}$ be crystal weights \eqref{e:weights} and suppose that $(w_n)_{n\geq0}$ is eventually decreasing. Then the associated Floquet function $h$ satisfies $h\in C^0(\mathbb S^1,\mathbb R)\cap C^1((0,1),\mathbb R)$. In particular, by symmetry, $h'(1/2)=0$.
\end{theorem}

\begin{proof}
We have $h(\theta) := \sum_{n \in \mathbb{Z}} w_n e^{2 i \pi n \theta}$ for $\theta \in \mathbb{R}$. Since $\sum_n w_n = 1$ and $w_n \geq 0$, this function is continuous, bounded, periodic on $\mathbb{R}$. To see that it is differentiable on $\mathbb{R} \setminus \mathbb{Z}$, we do an Abel transform. Rewrite $h$ as 
$$ h(\theta) = w_0 + 2\sum_{n=1}^\infty w_n (-1)^n + 2 \sum_{n=1}^\infty w_n (\cos(2 \pi n \theta) - (-1)^n). $$ 
Then, notice that, for all $N \geq 1$, we have (denoting $\varepsilon(n) = n w_n$)
$$  \sum_{n=1}^N w_n (\cos(2 \pi n \theta)-(-1)^n) = \sum_{n=1}^N \varepsilon(n) \frac{\cos(2 \pi n \theta)-(-1)^n}{n} $$ $$= \varepsilon(N+1) S_N(\theta) - \sum_{n=1}^{N} (\varepsilon(n+1)-\varepsilon(n)) S_n(\theta) $$
where $S_n(\theta) = \sum_{j=1}^n \frac{\cos(2 \pi j \theta)-(-1)^j}{j} $. Notice that, for $\theta \in (0,1)$, we have $$S_n(\theta) = - \sum_{j=1}^n \frac{1}{2 \pi }\int_{1/2}^\theta \sin(2 \pi j \nu) d\nu = -\frac{1}{2 \pi }\int_{1/2}^\theta  \sum_{j=1}^n \sin(2 \pi j \nu) d\nu = -\int_{1/2}^\theta  s_n(\nu) d\nu, $$
where $$s_n(\nu) := \frac{1}{2\pi} \sum_{j=1}^n \sin(2 \pi j \nu) =   \frac{ \sin(\pi n \nu) \sin(\pi(n+1)\nu)}{2 \pi  \sin(\pi \nu)} .$$
We have $s_n(\theta) \leq \frac{C}{\theta} + \frac{C}{1-\theta} $ for some $C \leq 1$ independent on $n$ and $\theta \in (0,1)$. Hence the same bound holds for $S_n(\theta)$. Added with the fact that $\varepsilon(n)-\varepsilon(n+1)$ is summable, we get convergence locally uniformly of the series:
$$ \forall \theta \in (0,1), \quad h(\theta) = w_0 + 2\sum_{n=1}^\infty w_n (-1)^n + 2\sum_{n=1}^\infty \left(\varepsilon(n)-\varepsilon(n+1) \right) S_n(\theta) .$$
The bound on $s_n(\theta)$ allows us to exchange the sum and the integral and we find
$$ \forall \theta \in (0,1), \quad h(\theta) =  w_0 + 2\sum_{n=1}^\infty w_n (-1)^n + 2 \int_{1/2}^\theta \sum_{n=1}^\infty \left(\varepsilon(n)-\varepsilon(n+1) \right) s_n(\nu) d\nu . $$
In particular, $h \in C^{1}((0,1),\mathbb{R})$, and for any $\theta \in (0,1)$, we have
$$ h'(\theta) = 2 \sum_{n=1}^\infty \left(\varepsilon(n)-\varepsilon(n+1)\right) s_n(\theta) .$$
\end{proof}

\begin{corollary}\label{cor:monotone}
Let $H_{\mathbf{w}}$ be an adjacency operator of a crystal as in \eqref{e:hw},\eqref{e:weights}. If the crystal weights $(w_k)_{k\geq0}$ are eventually decreasing, then (i) dispersion cannot be faster than $1/t$, and (ii) the crystal has an interval of purely absolutely continuous spectrum.
\end{corollary}

\begin{proof}
Claim (i) follows by combining Theorem~\ref{thm:monotone-regularity} with Proposition~\ref{thm:maxdis} for $\alpha=1$. For (ii), note that 
since $h$ is nonconstant and $C^1$ on $(0,1)$, there exists an interval $I$ on which $h'\neq0$ (indeed, if $h(a)\neq h(b)$, there is $c$ with $h'(c)=\frac{h(b)-h(a)}{b-a}\neq 0$ and by continuity of $h'$, there is an interval around $c$ where $h'(\theta)\neq 0$). It follows that $\sigma(H_{\mathbf{w}})\cap I$ is purely absolutely continuous, see e.g. \cite[Thm 6.3]{KernerPostSabriTaufer}.
\end{proof}
It is unclear if the previous bound is sharp. When the Floquet function is in $W^{1,2}$, we can do better, using the following argument from Kahane \cite[p.76]{Kahane70}, which is reproduced here for convenience (fixing typos).
\begin{lemma}\label{lem:kahane}
    If the Floquet function \eqref{e:hthetadef}-\eqref{e:weights} satisfies $h\in W^{1,2}(\mathbb{T})$, then dispersion cannot be faster than $t^{-1/2}$. This applies in particular to locally finite crystals. 
\end{lemma}
\begin{proof}
    Let $g_t=e^{ith}$. For any $N$ we have 
    \[
    \sum_{|m|>N}|\widehat{g_t}(m)|\le \Big(\sum_{|m|>N}m^{-2}\Big)^{1/2}\Big(\sum_{|m|>N}m^2|\widehat{g_t}(m)|^2\Big)^{1/2}\le \Big(\int_{|x|\ge N} x^{-2}dx\Big)^{1/2}\|g_t'\|_{L^2} = t \sqrt{\frac{2}{N}}\|h'\|_{L^2}.
    \]
    On the other hand, $\sum_{|m|\le N} |\widehat{g_t}(m)|\le (2N+1)^{1/2}(\sum_{|m|\le N}|\widehat{g_t}(m)|^2)^{1/2}=(2N+1)^{1/2}\|g_t\|_{L^2}=\sqrt{2N+1}$. This shows that $\|\widehat{g_t}\|_{\ell^1(\mathbb{Z})}\le t\sqrt{\frac{2}{N}}\|h'\|+\sqrt{2N+1}$. Choose $N=\lfloor t\rfloor$. Then $\|\widehat{g_t}\|_{\ell^1(\mathbb{Z})}\le C\sqrt{t}$ for large $t$.\\
    Finally, $1=\|g_t\|_{L^2}^2=\|\widehat{g_t}\|_{\ell^2}^2\le \|\widehat{g_t}\|_{\ell^1}\|\widehat{g_t}\|_{\ell^\infty}\le C\sqrt{t}\|\widehat{g_t}\|_{\ell^\infty}$. This proves that $\|\widehat{g_t}\|_{\ell^\infty}\ge C^{-1}t^{-1/2}$ as required.
\end{proof}
This statement is sharp as it was shown in \cite{KernerPostSabriTaufer} that the crystal with $w(k)=k^{-2}$, $k\neq 0$ disperses at speed $t^{-1/2}$. It might be non-sharp in finite range, where $t^{-1/3}$ is the fastest currently recorded. Lemma~\ref{lem:kahane} also applies to the Weierstrass crystal in the regime $\lambda\mu<1$, showing that the dispersion we get there cannot be faster than $t^{-1/2}$.

\section{A \texorpdfstring{$C^1$}{C1} Van Der Corput Lemma}\label{ap:A}

In this section we present two novel van der Corput lemmas in \S\,\ref{sec:2newvdc} which hold under weaker assumptions on the phase function than previously known, then we compare them to the literature in \S\,\ref{sec:earlyvdc}.

\subsection{Two novel vdC estimates}\label{sec:2newvdc}
Our first theorem has the weakest assumptions but yields no decay rate. Under an additional assumption, Theorem~\ref{thm:c1+alphavandercorput} will yield a polynomial rate. 

\begin{theorem}[$C^{1}$ Van Der Corput]\label{thm:c1vandercorput}
Let $h:[0,1] \rightarrow \mathbb{R}$ be a $C^{1}$ function. Suppose that the level sets of $h'$ have zero Lebesgue measure:
$$ \forall y \in \mathbb{R}, \qquad \ |\{ \theta \in [0,1]:  h'(\theta)= y \}| =0. $$
Then, for any continuous function $\chi:[0,1] \rightarrow \mathbb{C}$, we have:
$$   \sup_{m \in \mathbb{R}} \Big| \int_{0}^1 e^{2 i \pi (\xi h(\theta) - m \theta)} \chi(\theta) d\theta \Big| \xrightarrow{|\xi| \rightarrow \infty} 0. $$
\end{theorem}
\begin{proof}
Choose $\omega:(0,1] \rightarrow (0,\infty)$ a concave modulus of continuity for $\chi$ and $h'$ (one can choose for example the concave shell of any common modulus of continuity). By concavity, and since $\omega(0^+)=0^+$, $\omega$ is increasing, continuous and bijective near zero. Moreover, concavity yields $\omega(x/2) \geq \omega(x)/2$. For all $|\xi| \geq 3$, set $n(\xi) := \max \{ n \geq 0 \ | \ \xi^2 4^{-n} \omega(2^{-n}) \geq 1 \}$.  Obviously we have $n(\xi) \xrightarrow{|\xi| \rightarrow \infty}\infty$. Moreover, since $$ 1 \geq \xi^{2} 4^{-(n(\xi)+1)} \omega(2^{-(n(\xi)+1)}) \geq \xi^{2} 4^{-n(\xi)} \omega(2^{-n(\xi)})/8 \geq 1/8 ,$$
we find the growth relation
\begin{equation}\label{e:cr}
2^{-n(\xi)} \sqrt{\omega(2^{-n(\xi)})} \simeq 1/\xi .
\end{equation}
Now that this is fixed, let us start the proof. We let $|\xi| \geq 1$ and denote $n=n(\xi)$.
We will consider the numbers with finite binary digits: $$\forall \mathbf{a} \in \{0,1\}^{n}, \ x_\mathbf{a} := \sum_{k=1}^n a_k 2^{-k}.$$ We have:
$$\int_{0}^1 e^{2 i \pi (\xi h(\theta) - m \theta)} \chi(\theta) d\theta = \sum_{\mathbf{a} \in \{0,1\}^n} \int_{x_\mathbf{a}}^{x_{\mathbf{a}}+2^{-n}} e^{2 i \pi (\xi h(\theta) - m \theta)} \chi(\theta) d\theta  $$
$$ = 2^{-n} \sum_{\mathbf{a} \in \{0,1\}^n} \int_{0}^{1} e^{2 i \pi (\xi h(x_\mathbf{a}+\theta 2^{-n}) - m (x_\mathbf{a} + \theta 2^{-n}))} \chi(x_\mathbf{a}+\theta 2^{-n}) d\theta $$
$$ = 2^{-n} \sum_{\mathbf{a} \in \{0,1\}^n}  \chi(x_\mathbf{a}) e^{-2 i \pi m x_\mathbf{a}}\int_{0}^{1} e^{2 i \pi (\xi h(x_\mathbf{a}+\theta 2^{-n}) - m \theta 2^{-n})}  d\theta + \mathcal{O}(\omega(2^{-n})).$$
Now, since $h$ is $C^{1}$, we have by the mean value theorem, for all $\varepsilon$,
$$ h(x_{\mathbf{a}} + \varepsilon) = h(x_\mathbf{a})+h'(y(x_{\mathbf{a}},\varepsilon)) \varepsilon $$ for some $y \in [x_\mathbf{a},x_{\mathbf{a}}+\varepsilon]$, which implies the asymptotic expansion
$$ h(x_{\mathbf{a}} + \varepsilon) = h(x_\mathbf{a})+h'(x_\mathbf{a}) \varepsilon + \mathcal{O}(\varepsilon \  \omega(\varepsilon)), $$
from which we find that the previous sum of integrals is
$$  2^{-n} \sum_{\mathbf{a} \in \{0,1\}^n} \chi(x_\mathbf{a}) e^{-2 i \pi m x_\mathbf{a}}\int_{0}^{1} e^{2 i \pi (\xi h(x_\mathbf{a})+\xi h'(x_\mathbf{a}) \theta 2^{-n}  - m \theta 2^{-n})} e^{2 i \pi \xi \cdot \mathcal{O}(2^{-n} \omega(2^{-n}))} d\theta $$ $$ = 2^{-n} \sum_{\mathbf{a} \in \{0,1\}^n} \chi(x_\mathbf{a}) e^{2 i \pi (\xi h(x_\mathbf{a})-m x_\mathbf{a})} \int_{0}^{1} e^{2 i \pi \xi 2^{-n} ( h'(x_\mathbf{a}) - \frac{m}{\xi}) \theta} d\theta + \mathcal{O}\big(\sqrt{\omega(2^{-n})}\big) $$
using \eqref{e:cr}. Now, the inner integral is explicit. Denoting $\sigma^{-1} := \xi 2^{-n} \pi \simeq \omega(2^{-n})^{-1/2}$ and $\mu := m/\xi$, we compute:
$$ \Big|\int_{0}^{1} e^{2 i \sigma^{-1} ( h'(x_\mathbf{a}) - \mu) \theta} d\theta \Big| = \Big|\text{sinc}\Big(\frac{h'(x_\mathbf{a})-\mu}{\sigma} \Big)\Big| \leq \min\Big( 1 \ , \ \frac{\sigma}{|h'(x_\mathbf{a})-\mu|}  \Big). $$
We can then bound these integrals by either $1$ if $|h'(x_\mathbf{a})-\mu| \leq \sqrt{\sigma}$, or by $\sqrt{\sigma}$ if $|h'(x_\mathbf{a})-\mu| \geq \sqrt{\sigma}$.
Hence, we find the bound:
$$ \Big|\int_{0}^1 e^{2 i \pi (\xi h(\theta) - \mu \theta)} \chi(\theta) d\theta \Big| \leq  2^{-n} \sum_{\mathbf{a} \in \{0,1\}^n} \|\chi\|_\infty \Big| \int_{0}^{1} e^{2i  \sigma^{-1} ( h'(x_\mathbf{a}) - \mu) \theta} d\theta \Big| + \mathcal{O}\big(\sqrt{\omega(2^{-n})}\big) $$
$$ \lesssim \frac{1}{2^n}  \sum_{\mathbf{a} \in \{0,1\}^n} \min\Big( 1 \ , \ \frac{\sigma}{|h'(x_\mathbf{a})-\mu|}  \Big) + \mathcal{O}\big(\sqrt{\omega(2^{-n})}\big) $$
$$ \leq \frac{1}{2^n} \# \{ \mathbf{a} \in \{0,1\}^n, \ |h'(x_\mathbf{a}) - \mu| \leq \sqrt{\sigma} \} +   \mathcal{O}\big(\sqrt{\omega(2^{-n})}\big) .$$
To conclude, we use the regularity of $h'$ as follows. We have, for all $\theta \in [0, 2^{-n}]$, the bound 
$$ |h'(x_\mathbf{a}) - h'(x_\mathbf{a}+\theta)| \leq \omega(\theta) \le \omega(2^{-n})\,. $$
In particular, for $\mathbf{a} \in \{0,1\}$, we have:
$$ \left(|h'(x_\mathbf{a}) - \mu| \leq \sqrt{\sigma} \right) \implies \left( \forall \theta \in [0,2^{-n}], \ |h'(x_\mathbf{a}+\theta)-\mu| \leq \sqrt{\sigma} + \omega(2^{-n}) \right). $$
Which allows us to bound:
$$ \frac{1}{2^n} \# \{ \mathbf{a} \in \{0,1\}^n, \ |h'(x_\mathbf{a}) - \mu| \leq \sqrt{\sigma} \} = \mathrm{Leb}\bigg( \bigcup_{\substack{\mathbf{a} \in \{0,1\}^n\\{|h'(x_\mathbf{a}) - \mu|  \leq \sqrt{\sigma}}}} [x_\mathbf{a},x_\mathbf{a}+2^{-n}] \bigg) $$
$$ \leq |\{\theta \in [0,1]: |h'(\theta)-\mu| \leq C \omega(2^{-n})^{1/4}  \}|. $$
Now the pushforward measure $(h')_* d\theta$ is non-atomic by hypothesis, and is of compact support. It follows that its cumulative distribution function is uniformly continuous, with modulus of continuity given by some $\tilde{\omega}$. We can then bound
$$ |\{ \theta \in [0,1]: |h'(\theta)-\mu| \leq C \omega(2^{-n})^{1/4}  \}| \leq \tilde{\omega}(C \omega(2^{-n})^{1/4}). $$
by our nonlinearity hypothesis on $h'$.
In the end, our oscillatory integral is bounded by:
$$ \Big|\int_{0}^1 e^{2 i \pi (\xi h(\theta) - m \theta)} \chi(\theta) d\theta \Big| \lesssim \tilde{\omega}(C \omega(2^{-n})^{1/4}) + \sqrt{\omega(2^{-n})}.$$
Which proves that $\Big|\int_{0}^1 e^{2 i \pi (\xi h(\theta) - m \theta)} \chi(\theta) d\theta \Big| \xrightarrow{|\xi| \to \infty} 0$, and this uniformly in $m$.
\end{proof}

As a first application, suppose $(H_{\mathbf{w}}f)(n) = \sum_{k=0}^\infty w_k f(k+n)$ is a lacunary crystal, in the sense that $w_{\pm \lambda_j}=\mu_j\ge 0$, $w_k=0$ otherwise, where $\frac{\lambda_{j+1}}{\lambda_j}\ge q>1$. The Weierstrass crystal is a special case. Then we have:

\begin{corollary}\label{cor:dis}
Any lacunary crystal as above satisfies the dispersive limit $\|e^{it H_{\mathbf{w}}}\|_{\ell^1\to \ell^\infty}\to 0$, provided that $\sum_j \mu_j\lambda_j<\infty$.
\end{corollary}
\begin{proof}
    The Floquet function is $h(\theta)=\sum_j \mu_j \cos2\pi\lambda_j\theta$. The assumption guarantees that $h$ is $C^1$ with $h'(\theta)=-2\pi\sum_j\mu_j\lambda_j\sin 2\pi\lambda_j\theta$. Since $h'$ is itself a lacunary series, by \cite[Vol II. p. 267]{Bary}, all its level sets have zero measure.
\end{proof}

For Weierstrass crystals, this covers the regime $\lambda\mu<1$. But of course our main Theorem~\ref{thm:main} implies a stronger result.

The proof of Theorem~\ref{thm:c1vandercorput} implies an explicit rate of decay depending on the Frostman exponent of $(h')_* \lambda$ and the modulus of continuity of $\chi$ and $h'$. Recall that a measure $\mu$ on a metric space is a \emph{Frostman measure} if $\mu(B(x,r))\le Cr^s$ for any ball $B(x,r)$ of radius $r$. This is sometimes also called a \emph{H\"older continuous measure}. In the setting where $h \in C^{1+\alpha}$, we get polynomial decay under a natural hypothesis.

\begin{theorem}[$C^{1+\alpha}$ Van Der Corput]\label{thm:c1+alphavandercorput}
Let $\alpha \in (0,1)$ and let $h:[0,1] \rightarrow \mathbb{R}$ be a $C^{1+\alpha}$ function. Suppose that $h$ is sufficiently nonlinear, in the following sense:
$$ \exists\, C_0, \ \exists\, \gamma \in (0,1), \ \forall I \subset \mathbb{R} \text{ interval}, \quad  |\{ \theta \in [0,1]: h'(\theta) \in I \}| \leq C_0 |I|^\gamma .$$
Then there exists a constant $C \geq 1$ that depends only on $C_0$ such that, for any $\alpha$-Hölder function $\chi:[0,1] \rightarrow \mathbb{C}$, we have:
$$  \forall\, |\xi| \geq 1, \qquad \sup_{m \in \mathbb{R}} \Big| \int_{0}^1 e^{2 i \pi (\xi h(\theta) - m \theta)} \chi(\theta) d\theta \Big| \leq C \|h\|_{C^{1+\alpha}}  \|\chi\|_{C^\alpha} |\xi|^{-\alpha \gamma/6}. $$
\end{theorem}

\begin{proof}
Recall that the previous proof implies us a rate of decay of the form
$$ \sup_m \Big| \int_0^1 e^{2 i \pi (\xi h-m\theta)} \chi d\theta \Big| \lesssim \tilde{\omega}\Big(C \omega(2^{-n(\xi)})^{1/4} \Big) + \omega(2^{-n(\xi)})^{1/2}$$
where $n(\xi)$ is fixed so that $\xi^2 4^{-n(\xi)} \omega(2^{-n(\xi)}) \simeq 1$.
In our current setting, we have $\tilde{\omega}(t) \simeq t^\gamma$ and $\omega(t) \simeq t^\alpha$, so that
$$ \xi^2 4^{-n(\xi)} \omega(2^{-n(\xi)}) \simeq \xi^{2} 2^{-(2+\alpha)(n(\xi))} \simeq 1$$
implies $ 2^{-n(\xi)} \simeq \xi^{-\frac{2}{2+\alpha}} $,
and the bound becomes
$$ \tilde{\omega}\left(C \omega(2^{-n(\xi)})^{1/4} \right) + \omega(2^{-n(\xi)})^{1/2} \lesssim \xi^{-\frac{2\alpha \gamma}{4(2+\alpha)}} \lesssim \xi^{- \alpha \gamma/6}, $$
which is the announced bound.
\end{proof}

Using Lemma~\ref{lem:analfro}, we get in particular,

\begin{corollary}\label{cor:a3}
Any crystal associated to an analytic non-constant Floquet function exhibits polynomial dispersion.
\end{corollary}

\begin{corollary}
    Any locally finite crystal exhibits polynomial dispersion. The same holds for long range crystals with exponentially decaying weights.
\end{corollary}

\subsection{Comparison with earlier vdC}\label{sec:earlyvdc}
    The classical van der Corput lemma says that if $g\in C^k$, $|g^{(k)}(\theta)|\ge c>0$ and if $g'$ is moreover monotone if $k=1$, then
    \begin{equation}\label{e:clasvdc}
    \Big|\int_a^b e^{i\xi g(\theta)}\chi(\theta)d\theta\Big|\le c_{k}|\xi|^{-1/k}\Big(|\chi(b)|+\int_a^b|\chi'(\theta)|d\theta\Big) \,.
    \end{equation}

    Classical sublevel set estimates show that $|g^{(k)}(\theta)|\ge c\implies |\{\theta\in I:g'(\theta)\in I\}|\le C_k(|I|/c)^{1/(k-1)}$. So if $k\ge 2$ our Theorem~\ref{thm:c1+alphavandercorput} applies and implies a slightly weaker exponent $\frac{1}{6(k-1)}$ rather than $\frac{1}{k}$. For $k=1$, \eqref{e:clasvdc} is not suited for uniform estimates in $m$, because if one takes $g(\theta)=h(\theta)-m\theta$, then $|g'(\theta)|$ is never bounded from below (choose $m=\xi h'(\theta_0)$). And if one takes $\chi(\theta)=e^{2\pi i m\theta}$ then \eqref{e:clasvdc} is unbounded in $m$.

    The recent result of Algom-Chang-Wu-Wu \cite{ACWW} holds under a non-flatness condition. Let $N_r(A)$ be the minimum number of intervals of length $r$ that can cover a set $A$. For example, if $A=[0,b]$ then $N_r(A)=\lceil \frac{b}{r}\rceil$. Let $0<\delta\le 1$. We say that $f$ is \emph{$\delta$-nonflat} in an interval $J$ if for all $\varepsilon>0$, for all large enough $n \geq n_0(\varepsilon)$ and for every bounded $A\subset J$, we have 
    \begin{equation}\label{e:acww}
    \mathrm{diam}f(A)<2^{-n}\implies N_{2^{-\delta n}}(A)<2^{\varepsilon n} \,.
    \end{equation}
    This says that if $f$ is nearly flat on $A$, then $A$ must be geometrically small. It is shown in \cite{ACWW} that if $h\in C^{1+\alpha}(J)$ for some $0<\alpha\le 1$, if $h'$ is $\delta$-nonflat on $J$, and if $\nu$ is a self-similar nonatomic measure, then there are $\tau(\nu,\delta,\alpha)$ and $C(h)$ such that $|\int e^{2\pi i \lambda g(\theta)}d\nu(\theta)|\le C\lambda^{-\tau}$.

    The main significance of \cite{ACWW} is that it holds for general measures $\nu$. On the other hand, for the special case of the Lebesgue measure, our assumptions are weaker. In fact, if \eqref{e:acww} holds for $f=h'$, let $A=f^{-1}(I)$ with $|I|= 2^{-n-1}$. Then $\mathrm{diam}f(A)\le 2^{-n-1}$, so by \eqref{e:acww}, we get $|f^{-1}(I)|=|A|\le N_{2^{-\delta n}}(A)2^{-\delta n}\le 2^{-(\delta-\varepsilon)n}\le C_\gamma |I|^\gamma$ for any $\gamma<\delta$. So our Theorem~\ref{thm:c1+alphavandercorput} applies, and implies moreover an explicit rate of decay, and is also uniform in $m$.

    Summarizing, the assumptions of Theorem~\ref{thm:c1+alphavandercorput} are weaker than the assumptions of the classical van der Corput lemma, and also weaker than the assumptions of \cite{ACWW} for the particular case of the Lebesgue measure, and Theorem~\ref{thm:c1vandercorput} has the weakest assumptions.

    Let us now give examples where our results have an advantage:
    \begin{itemize}
        \item[(i)] Corollary~\ref{cor:dis} implies a simple illustration where Theorem~\ref{thm:c1vandercorput} directly applies, while it is generally difficult to check for extra assumptions such as those of Theorem~\ref{thm:c1+alphavandercorput}.
        \item[(ii)] Let $f(x)=e^{-1/x^2}$ if $x>0$, $f(0)=0$ and $h(\theta)=\int_0^\theta f(x)dx$. Then Theorem~\ref{thm:c1vandercorput} applies. However, $\{\theta\in [0,1]:h'(\theta)\le r\}=[0,\frac{1}{\sqrt{\ln(1/r)}}]$, so $|\{\theta:h'(\theta)\in [0,r]\}|=\frac{1}{\sqrt{\ln(1/r)}}$. Hence, Theorem~\ref{thm:c1+alphavandercorput} does not apply, and neither do the classical vdC or \cite{ACWW}.
        \item[(iii)] Let $f$ be either a fractional Brownian motion $f=B_H$, $H\in (0,1)$, or $f=W_{\mu,\lambda}$, with $\alpha=\frac{|\ln\mu|}{\ln\lambda}<1$, and take $h(\theta)=\int_0^\theta f(x)dx$. Theorem~\ref{thm:c1+alphavandercorput} applies to both cases. In fact, $B_H$ is $\beta$-H\"older continuous for $\beta<H$, and its local time $L$ exists and is continuous, hence bounded on the compact $[0,1]$, and we get $|\{\theta:h'(\theta)\in I\}|=\int_I L(y)dy\le M|I|$. Theorem~\ref{thm:c1+alphavandercorput} thus implies us a dispersive decay $t^{-\sigma}$ for any $\sigma<\frac{H}{6}$. Similarly, in case of $f = W_{\mu,\lambda}$, $f$ is $\alpha$-H\"older, but we skip the lengthy proof that its occupation measure is Frostman, i.e. that $f'=W_{\mu,\lambda}$ satisfies the non-concentration bound required in Theorem~\ref{thm:c1+alphavandercorput}. Clearly classical vdC applies to neither example: $h'$ is not monotone and $h''$ does not exist. In fact, the results of \cite{ACWW} do not apply either. To see this, we show that if $f$ is $\delta$-nonflat, then its graph $\mathcal{G}\subset [0,1]\times\mathbb{R}$ has Hausdorff dimension $1$. This is sufficient, since the graph of $f=B_H$ has dimension $2-H$ a.s. \cite{Adler}, and the graph of $f=W_{\mu,\lambda}$ has dimension $2-\frac{|\ln\mu|}{\ln \lambda}$  \cite{RenShen}, so neither $f$ can be $\delta$-nonflat.

        So let $f$ be $\delta$-nonflat, fix $s>1$, and let us show that $\mathcal{H}^s(\mathcal{G})=0$. Set $\varepsilon := (s-1)/2$. Without loss of generality, $f([0,1]) \subset [0,1]$. We cut $[0,1]$ into dyadic intervals $(I_i)$ of size $2^{-n-1}$. Then by \eqref{e:acww}, for each $i$, the set $f^{-1}(I_i)$ can be covered by $2^{\varepsilon n}$ intervals $(J_{i,j})_j$ of length $2^{-\delta n}$. So we can cover the graph $\mathcal{G}$ by the sets $(J_{i,j} \times I_i)$, which give a covering of cardinality at most $2^{(1+\varepsilon)n}$ with each element of diameter $\leq 2^{-n+1}$. It follows that $\mathcal{H}^s(\mathcal{G}) \leq \sum_{i,j} \text{diam}(J_{i,j} \times I_i)^s \lesssim 2^{(1+\varepsilon-s) n} \rightarrow 0$. Thus, $\dim_H(\mathcal{G})\le 1$. Since we always have $\dim_H(\mathcal{G})\ge 1$, the claim is proved.
    \end{itemize}

\appendix

\section{On the large lacunarity limit} \label{ap:large-lacunarity}

Fix $0<\mu<1$, and let $\varphi\in C^0(\mathbb S^1,\mathbb R)$. Suppose that there exists an interval $I$ on which $\varphi$ is analytic and non-constant. For each integer $\lambda\geq2$, let $W_{\mu,\lambda}(\theta)=\sum_{n=0}^\infty\mu^n\varphi(\lambda^n\theta)$. Since increasing the lacunarity produces arbitrarily fast Fourier decay and dispersion, it is natural to ask what the limiting object is as $\lambda\to\infty$. We show that in this limit the terms $\varphi(\lambda^n\theta)$ behave asymptotically as independent random variables. We study more precisely the limit of the occupation measure of $W_{\mu,\lambda}$, which is also the spectral measure of $H_{\mathbf{w}}$ at $\delta_0$.
%the limiting measure below is the law of the random series $\sum_{n=0}^\infty\mu^nX_n$, where $(X_n)_{n\geq0}$ are i.i.d.\ with common distribution $\varphi_*d\theta$. Thus, in the large-lacunarity limit, the deterministic Weierstrass series behaves like a random series with i.i.d.\ coefficients.

\begin{theorem}\label{thm:largelim}
Under the assumptions above, for any $0 < \mu < 1$, the weak limit
\[
m_\mu:=\lim_{\lambda\to\infty}(W_{\mu,\lambda})_*(d\theta)
\]
exists and is the distribution of the random series $\sum_{n=0}^\infty\mu^nX_n$, where the $X_n$ are i.i.d.\ with distribution $\varphi_*d\theta$. Moreover, $|\widehat{m_\mu}(\xi)|\lesssim|\xi|^{-\rho}$ as $|\xi|\to\infty$ for some $\rho>0$.
\end{theorem}
The result applies in particular to functions $\varphi$ taking predominantly the values $\pm1$, e.g. $\varphi=1$ on $(\frac{1}{10},\frac{4}{10})$, $\varphi=-1$ on $(\frac{6}{10},\frac{9}{10})$ and piecewise affine on the rest of the circle, so that the resulting function is continuous. The corresponding limit is closely related to a Bernoulli convolution. 
\begin{proof}[Proof of Theorem~\ref{thm:largelim}]
The space of probability measures on a compact space, equipped with the topology of weak convergence, is itself a compact metric space (for example, using the Lévy–Prokhorov metric). In particular, up to a subsequence, a weak limit $m_\mu$ always exists. Let us show that the limit is unique. We have, for a fixed $t \in \mathbb{R}$:
\begin{align*} \widehat{W_* (d\theta)}(t) = \int_{\mathbb{S}^1} e^{i t W_{\mu,\lambda}(\theta)} d\theta = \int_{\mathbb{S}^1} e^{i t \varphi(\theta)} e^{it \mu W_{\mu,\lambda}(\lambda \theta)} d\theta = \int_{\mathbb{S}^1} \mathcal{L}_\lambda( e^{i t \varphi})(\theta) e^{it \mu W_{\mu,\lambda}(\theta)} d\theta.
\end{align*}
Now, $e^{i t \varphi}$ is uniformly continuous. If we denote by $\varepsilon_t(\cdot)$ a modulus of continuity, we can write
$$ \Big| \mathcal{L}_\lambda(e^{it \varphi})(\theta) - \int_{\mathbb{S}^1} e^{i t \varphi} d\theta \Big| \leq \varepsilon_t(1/\lambda) .$$
Hence
\[
\widehat{W_*(d\theta)}(t) = \int_{\mathbb{S}^1} e^{it \varphi} \int_{\mathbb{S}^1} e^{it\mu W_{\mu,\lambda}} + O(\varepsilon_t(1/\lambda))= F(t) \widehat{W_* (d\theta)}(\mu t) + \mathcal{O}(\varepsilon_t(1/\lambda)),
\]
where $F(t) := \int_{\mathbb{S}^1} e^{i t \varphi}$. Taking the limit of $\lambda \rightarrow +\infty$ along our subsequence yields the relation $ \widehat{m_\mu}(t) = F(t)\widehat{m_\mu}(t\mu) $, so that
$$ \widehat{m_\mu}(t) = \prod_{n=0}^\infty F(t \mu^n). $$
If $(X_n)_{n\geq0}$ are i.i.d.\ with distribution $\varphi_*d\theta$, then $F(t)=\mathbb E(e^{itX_0})$, and by independence
$$ \mathbb E\left(e^{it\sum_{n=0}^\infty\mu^nX_n}\right)=\prod_{n=0}^\infty F(t\mu^n). $$
Hence $m_\mu$ is precisely the distribution of $\sum_{n=0}^\infty\mu^nX_n$. This also proves that the limit is unique, so the weak limit indeed exists.

Now let us bound $F$. We have
$$ F(t) = \int_{\mathbb{S}^1} e^{it \varphi(\theta)} d\theta = F_1(t) + F_2(t) $$
where $F_1(t) = \int_{I} e^{it \varphi}$, in the region where $\varphi$ is analytic and non-constant, and $F_2(t) = \int_{\mathbb{S}^1 \setminus I} e^{it \varphi}$ is the remaining part. We have the bounds $F_1(t) \lesssim t^{-\rho}$ for some $\rho$, as follows from the standard van der Corput lemma or using Corollary~\ref{cor:a3}, and $|F_2(t)| \leq |\mathbb{S}^1 \setminus I| := 1-|I|$. In particular, if $t$ is large enough $t \geq T$, we can write $|F(t)| \leq 1-|I|/2$. Then, letting $N(t)$ be the largest integer such that $t \mu^{N(t)} \geq T$, in other words $N(t) := \lfloor \ln(T/t)/\ln(\mu) \rfloor$, we find
$$ |\widehat{m_\mu}(t)| \leq \prod_{n=1}^{N(t)} |F(t \mu^n)| \leq (1-|I|/2)^{N(t)} \leq (1-|I|/2)^{-1} (t/T)^{- |\ln(1-|I|/2)/\ln(\mu)|} \lesssim t^{-\rho} $$
where $\rho := |\frac{\ln(1-\frac{|I|}{2})}{\ln\mu}|$.
\end{proof}

\section*{AI statement}

AI tools were used only for language editing, cross-checking the final manuscript and literature searches. All original ideas and proofs are completely those of the authors, who checked the final manuscript and take full responsibility for its contents.

\end{document}